\documentclass[12pt]{amsart}
\usepackage{graphicx}
\usepackage{pslatex}
\usepackage{amsmath,amsfonts}
\usepackage{rotating}
\usepackage{amssymb}
\usepackage{verbatim}
\usepackage{rotating}
\usepackage{color}
\usepackage{mathrsfs}
\usepackage{hyperref}

\newtheorem{theorem}{Theorem}[section]
\newtheorem{lemma}[theorem]{Lemma}
\newtheorem{proposition}[theorem]{Proposition}
\newtheorem{corollary}[theorem]{Corollary}

\theoremstyle{definition}

\theoremstyle{remark}

\numberwithin{equation}{section}

\def\ds{\displaystyle}

\def\R{\mathbb R}
\def\C{\mathbb C}
\def\Z{\mathbb Z}
\def\N{\mathbb N}
\def\S{\mathscr S}

\def\supp{\text{supp}}
\def\({\left(}
\def\){\right)}
\def\[{\left[}
\def\]{\right]}
\def\<{\left<}
\def\>{\right>}
\def\less{\lesssim}

\begin{document}

\title{Hardy Space Estimates for Littlewood-Paley-Stein Square Functions and Calder\'on-Zygmund Operators}

\author{Jarod Hart}
\address{Department of Mathematics\\ Wayne State University\\ Detroit, MI 48202}
\email{jarod.hart@wayne.edu}
\author{Guozhen Lu}
\address{Department of Mathematics\\ Wayne State University\\ Detroit, MI 48202}
\email{gzlu@wayne.edu}

\thanks{Hart was partially supported by an AMS-Simons Travel Grant.  Lu was supported by NSF grant \#DMS1301595   .}



\date{\today}

\dedicatory{ }


\begin{abstract}
In this work, we give new sufficient conditions for Littlewood-Paley-Stein square function and necessary and sufficient conditions for a Calder\'on-Zygmund operator to be bounded on Hardy spaces $H^p$ with indices smaller than $1$.  New Carleson measure type conditions are defined for Littlewood-Paley-Stein operators, and the authors show that they are sufficient for the associated square function to be bounded from $H^p$ into $L^p$.  New polynomial growth $BMO$ conditions are also introduced for Calder\'on-Zygmund operators.  These results are applied to prove that Bony paraproducts can be constructed such that they are bounded on Hardy spaces with exponents ranging all the way down to zero.

\end{abstract}

\maketitle

\section{Introduction}

The purpose of this work is to prove new Hardy space $H^p(\R^n)$ bounds for Littlewood-Paley-Stein square functions and Calder\'on-Zygmund integral operators where the index $p$ is allowed to be small.  Part of the novelty of the work here is that it draws an explicit connection between Calder\'on-Zygmund operators and Littlewood-Paley-Stein square functions.    

It is well known by now that one way to define the real Hardy spaces $H^p$ for $0<p<\infty$ is by using certain convolution-type Littlewood-Paley-Stein square functions.  This has been explored by many mathematicians; some of the fundamental developments of this idea can be found in the work of Stein \cite{St1,St2} and Fefferman and Stein \cite{FS2}.  In particular, Fefferman and Stein proved that one can define $H^p=H^p(\R^n)$ using square functions of the form
\begin{align*}
S_Qf(x)=\(\sum_{k\in\Z}|Q_kf(x)|^2\)^\frac{1}{2},
\end{align*}
associated to integral operators $Q_kf=\psi_k*f$ for an appropriate choice of Schwartz function $\psi\in\S$, where $\psi_k(x)=2^{kn}\psi(2^kx)$.  There are also results in the direction of determining the most general classes of such convolution operators that can be used to define Hardy spaces, or more generally Triebel-Lizorkin spaces; see for example the work of Bui, Paluszy\'nski, and Taibelson \cite{BPT1,BPT2}.  Generalized classes of non-convolution type Littlewood-Paley-Stein square function operators were studied, for example, in \cite{DJ,DJS,Se}.  Although all of the bounds in these articles are relegated to Lebesgue spaces with index $p\in(1,\infty)$, which for this range of indices coincide with Hardy spaces.  In the current work, we consider a general class of non-convolution type Littlewood-Paley-Stein square function operators acting on Hardy spaces with indices smaller than $1$.

Before we state our Hardy space estimates for Littlewood-Paley-Stein square functions, we define our classes of Littlewood-Paley-Stein square function operators.  Given kernel functions $\lambda_k:\R^{2n}\rightarrow\C$ for $k\in\Z$, define
\begin{align*}
\Lambda_kf(x)=\int_{\R^{n}}\lambda_k(x,y)f(y)dy
\end{align*}
for appropriate functions $f:\R^n\rightarrow\C$.  Define the square function associated to $\{\Lambda_k\}$ by
\begin{align*}
S_\Lambda f(x)=\(\sum_{k\in\Z}|\Lambda_kf(x)|^2\)^\frac{1}{2}.
\end{align*}
We say that a collection of operators $\Lambda_k$ for $k\in\Z$ is a collection of Littlewood-Paley-Stein operators with decay $N$ and smoothness $L+\delta$, written $\{\Lambda_k\}\in LPSO(N,L+\delta)$, for $N>0$, an integer $L\geq0$ and $0<\delta\leq1$, if there exists a constant $C$ such that
\begin{align}
&|\lambda_k(x,y)|\leq C\,\Phi_k^{N}(x-y)\label{sqker1}\\
&|D_1^\alpha \lambda_k(x,y)|\leq C2^{|\alpha|k}\Phi_k^{N}(x-y)\text{ for all }|\alpha|=\alpha_1+\cdots+\alpha_n\leq L\label{sqker2}\\
&|D_1^\alpha \lambda_k(x,y)-D_1^\alpha \lambda_k(x,y')|\leq C|y-y'|^\delta \,2^{k(L+\delta)}\(\Phi_k^{N}(x-y)+\Phi_k^{N}(x-y')\)\text{ for all $|\alpha|=L$.}\label{sqker3}
\end{align}
Here we use the notation $\Phi_k^N(x)=2^{kn}(1+2^k|x|)^{-N}$ for $N>0$, $x\in\R^n$, and $k\in\Z$.  We also use the notation $D_0^\alpha F(x,y)=\partial_x^\alpha F(x,y)$ and $D_1^\alpha F(x,y)=\partial_y^\alpha F(x,y)$ for $F:\R^{2n}\rightarrow\C$ and $\alpha\in\N_0^n$.  It can easily be shown that $LPSO(N,L+\delta)\subset LPSO(N',L+\delta')$ for all $0<\delta'\leq\delta\leq1$ and $0<N'\leq N$.

Our goal in studying square functions of the form $S_\Lambda$ is to prove boundedness properties from $H^p$ into $L^p$.  Note that it is not reasonable to expect $S_\Lambda$ to be bounded from $H^p$ into $H^p$ when $0<p\leq1$ since $S_\Lambda f\geq0$.  It is also not hard to see that the condition $\{\Lambda_k\}\in LPSO(N,L+\delta)$ alone, for any $N>0$, $L\geq0$, and $0<\delta\leq1$, is not sufficient to guarantee that $S_\Lambda$ to be bounded from $H^p$ into $L^p$ for any $0<p<\infty$.  In fact, this is not true even in the convolution setting.  This can be seen by taking $\lambda_k(x,y)=\varphi_k(x-y)$ for some $\varphi\in \S$ with non-zero integral, where $\varphi_k(x)=2^{kn}\varphi(2^kx)$.  The square function $S_\Lambda$ associated to this convolution operator is not bounded from $H^p$ into $L^p$ for any $0<p<\infty$.  Hence some additional conditions are required for $\Lambda_k$ in order to assure $H^p$ to $L^p$ bounds.  For $1<p<\infty$, this problem was solved in terms of Carleson measure conditions on $\Lambda_k1(x)$; see for example \cite{C,J,CJ,Se}.  We give sufficient conditions for such bounds when the index $p$ is allowed to range smaller than $1$.  The additional cancellation conditions we impose on $\Lambda_k$ involve generalized moments for non-concolution operators $\Lambda_k$.  Define the moment function $[[\Lambda_k]]_\beta(x)$ by the following.  Given $\{\Lambda_k\}\in LPSO(N,L+\delta)$ and $\alpha\in\N_0^n$ with $|\alpha|<N-n$
\begin{align*}
\[\[\Lambda_k\]\]_{\alpha}(x)&=2^{k|\alpha|}\int_{\R^n}\lambda_k(x,y)(x-y)^\alpha dy
\end{align*}
for $k\in\Z$ and $x\in\R^n$.  It is worth noting that $[[\Lambda_k]]_0(x)=\Lambda_k1(x)$, which is a quantity that is closely related to $L^2$ bounds for $S_\Lambda$, see for example \cite{DJ,DJS,Se}.  We use these moment functions to provide sufficient conditions of $H^p$ to $L^p$ bounds for $S_\Lambda$ in the following theorem.

\begin{theorem}\label{t:sqbound}
Let $\{\Lambda_k\}\in LPSO(N,L+\delta)$, where  $N=n+2L+2\delta$ for some integer $L\geq0$ and $0<\delta\leq1$.  If
\begin{align}
d\mu_\alpha(x,t)=\sum_{k\in\Z}|[[\Lambda_k]]_{\alpha}(x)|^2\delta_{t=2^{-k}}\,dx\label{dmuCarleson}
\end{align}
is a Carleson measure for all $\alpha\in\N_0^n$ with $|\alpha|\leq L$, then $S_\Lambda$ can be extended to a bounded operator from $H^p$ into $L^p$ for all $\frac{n}{n+L+\delta}<p\leq1$.
\end{theorem}

Here we say that a non-negative measure $d\mu(x,t)$ on $\R^{n+1}_+=\R^n\times(0,\infty)$ is a Carleson measure if there exists $C>0$ such that $d\mu(Q\times(0,\ell(Q)))\leq C|Q|$ for all cubes $Q\subset\R^n$, where $\ell(Q)$ denotes the sidelength of $Q$.  We only prove a sufficient condition here for boundedness of $S_\Lambda$ from $H^p$ into $L^p$, but it is reasonable to expect that the Carleson measure conditions in \eqref{dmuCarleson} are also necessary.  We hope to resolve this issue entirely with a full necessary and sufficient condition in future work.  We also provide a quick corollary of Theorem \ref{t:sqbound} to the type of operators studied in \cite{DJ,DJS,Se}, among others.

\begin{corollary}\label{c:sqfunction}
Let $\{\Lambda_k\}\in LPSO(n+2\delta,\delta)$ and $0<\delta\leq1$.  If $S_\Lambda$ is bounded on $L^2$, then $S_\Lambda$ extends to a bounded operator from $H^p$ into $L^p$ for all $\frac{n}{n+\delta}<p\leq1$.
\end{corollary}

Corollary \ref{c:sqfunction} easily follows from Theorem \ref{t:sqbound} and the following observation.  If $S_\Lambda$ is bounded on $L^2$, then $d\mu_0(x,t)$, as defined in \eqref{dmuCarleson} for $\alpha=0$, is a Carleson measure; see \cite{C,J} for proof of this observation.

Another purpose of this work is to prove a characterization of Hardy space bounds for Calder\'on-Zygmund operators.  Some of the earliest development of singular integral operators on Hardy spaces is due to Stein and Weiss \cite{SW}, Stein \cite{St2}, and Feffermand and Stein \cite{FS2}.  It was proved by Fefferman and Stein \cite{FS2} that if $T$ is a convolution-type singular integral operator that is bounded on $L^2$, then $T$ is bounded on $H^p$ for $p_0<p<\infty$ where $0\leq p_0<1$ depends on the regularity of the kernel of $T$.  This situation is considerably more complicated in the non-convolution setting, which can be observed in the $T1$ type theorems in \cite{DJ,T,FTW,FHJW,AM}.  In the 1980's David and Journ\'e proved the celebrated $T1$ theorem that provided necessary and sufficient conditions for Lebesgue space $L^p$ bounds for non-convolution Calder\'on-Zygmund operators when $1<p<\infty$, which coincides with the Hardy space bounds for this range of indices.  In \cite{T,FTW,FHJW}, the authors give sufficient $T1$ type conditions for a Calder\'on-Zygmund operator to be bounded on $H^p$ for $0<p\leq1$.  The conditions in \cite{T,FTW,FHJW} are too strong though, in the sense that they are not necessary for Hardy space bounds.  The fact that the conditions in \cite{T,FTW,FHJW} are not necessary can be seen by the full necessary and sufficient conditions provided in \cite{AM} when $p_0<p\leq1$, where $p_0=\frac{n}{n+\gamma}$ and $\gamma$ is a regularity parameter for the kernel of $T$.  This can also be seen by considering the Bony paraproduct, which we prove (in Theorem \ref{t:Bonyparaproduct}) is bounded on $H^p$ for $p_0<p\leq1$ and $p_0$ can be taken arbitrarily close to zero.  One of the main purposes of this article is to prove at full necessary and sufficient $T1$ type theorem for Calder\'on-Zygmund operators on Hardy spaces (Theorem \ref{t:CZbound}), thereby generalizing results pertaining to $H^p$ bounds from \cite{FS2,AM,T,FHJW,FTW}.

We say that a continuous linear operator $T$ from $\S$ into $\S'$ is a Calder\'on-Zygmund operator with smoothness $M+\gamma$, for any integer $M\geq0$ and $0<\gamma\leq1$, if $T$ has function kernel $K:\R^{2n}\backslash\{(x,x):x\in\R^n\}\rightarrow\C$ such that
\begin{align*}
\<Tf,g\>=\int_{\R^{2n}}K(x,y)f(y)g(x)dy\,dx
\end{align*}
whenever $f,g\in C_0^\infty=C_0^\infty(\R^n)$ have disjoint support, and there is a constant $C>0$ such that the kernel function $K$ satisfies
\begin{align*}
&|D_0^\alpha D_1^\beta K(x,y)|\leq \frac{C}{|x-y|^{n+|\alpha|+|\beta|}}\text{ for all }|\alpha|,|\beta|\leq M,\\
&|D_0^\alpha D_1^\beta K(x,y)-D_0^\alpha D_1^\beta K(x',y)|\leq \frac{C|x-x'|^\gamma}{|x-y|^{n+M+|\beta|+\gamma}}\text{ for }|\beta|\leq|\alpha|=M,\text{  }|x-x'|<|x-y|/2,\\
&|D_0^\alpha D_1^\beta K(x,y)-D_0^\alpha D_1^\beta K(x,y')|\leq \frac{C|y-y'|^\gamma}{|x-y|^{n+|\alpha|+M+\gamma}}\text{ for }|\alpha|\leq|\beta|=M,\text{  }|y-y'|<|x-y|/2.
\end{align*}
We will also define moment distributions for an operator $T\in CZO(M+\gamma)$, but we require some notation first.  For an integer $M\geq0$, define the collections of smooth functions of polynomial growth $\mathcal O_M=\mathcal O_M(\R^n)$ and of smooth compactly supported function with vanishing moments $\mathcal D_M=\mathcal D_M(\R^n)$ by
\begin{align*}
&\mathcal O_M=\left\{f\in C^\infty(\R^n):\sup_{x\in\R^n}|f(x)|\cdot(1+|x|)^{-M}<\infty\right\}\text{ and}\\
&\mathcal D_M=\left\{f\in C_0^\infty(\R^n):\int_{\R^n}f(x)x^\alpha dx=0\text{ for all }|\alpha|\leq M\right\}.
\end{align*}
Let $\eta\in C_0^\infty(\R^n)$ be supported in $B(0,2)$, $\eta(x)=1$ for $x\in B(0,1)$, and $0\leq\eta\leq1$.  Define for $R>0$, $\eta_R(x)=\eta(x/R)$.  We reserve this notation for $\eta$ and $\eta_R$ throughout.  In \cite{T,FTW,FHJW}, the authors define $Tf$ for $f\in\mathcal O_M$ where $T$ is a linear singular integral operator.  We give an equivalent definition to the ones in \cite{T,FTW,FHJW}.  Let $T$ be a $CZO(M+\gamma)$ and $f\in\mathcal O_{M}$ for some integer $M\geq0$ and $0<\gamma\leq1$.  For $\psi\in C_0^\infty(\R^n)$, choose $R_0\geq1$ minimal so that $\supp(\psi)\subset\overline{B(0,R_0/4)}$, and define
\begin{align*}
&\<Tf,\psi \>=\lim_{R\rightarrow\infty}\<T(\eta_R\,f),\psi\> - \sum_{|\beta|\leq M}\int_{\R^{2n}}\frac{D_0^\beta K(0,y)}{\beta!}x^\beta(\eta_R(y)-\eta_{R_0}(y))f(y)\psi(x) dy\,dx.
\end{align*}
This limit exists based on the kernel representation and kernel properties for $T\in CZO(M+\gamma)$ and is independent of the choice of $\eta$, see \cite{T,FTW,FHJW} for proof of this fact.  The choice of $R_0$ here is not of consequence as long as $R_0$ is large enough so that $\supp(\psi)\subset\overline{B(0,R_0/4)}$; we choose it minimal to make this definition precise.  The definition of $\<Tf,\psi\>$ depends on $\psi$ here through the support properties of $\psi\in C_0^\infty$, but for $\psi\in\mathcal D_M$, it follows that $\<Tf,\psi \>=\lim_{R\rightarrow\infty}\<T(\eta_R\,f),\psi\>$ since the integral term above vanishes for such $\psi$.  Now we define the moment distribution $[[T]]_\alpha\in\mathcal D_M'$ for $T\in CZO(M+\gamma)$ and $\alpha\in\N_0^n$ with $|\alpha|\leq M$ by
\begin{align*}
\<[[T]]_\alpha,\psi\>&=\lim_{R\rightarrow\infty}\int_{\R^{2n}}\mathcal K(u,y)\psi(u)\eta_R(y)(u-y)^\alpha dy\,du
\end{align*}
for $\psi\in\mathcal D_{|\alpha|}$, where $\mathcal K\in\S'(\R^{2n})$ is the distribution kernel of $T$.  We abuse notation here in that the integral in this definition is not necessarily a measure theoretic integral; rather, it is the dual pairing between elements of $\S(\R^{2n})$ and $\S'(\R^{2n})$.  Throughout this work, we will use $\mathcal K$ to denote distributional kernels and $K$ to denote function kernels for Calder\'on-Zygmund operators.  When we write $\mathcal K$ in an integral over $\R^{2n}$, the integral is understood to be a the pairing of $\mathcal K\in\S'(\R^{2n})$ with an element of $\S(\R^{2n})$.  It is not hard to show that this definition is well-defined by techniques from \cite{T,FTW,FHJW}.  This distributional moment associated to $T$ generalizes the notion of $T1$ as used in \cite{DJ} in the sense that $\<[[T]]_0,\psi\>=\<T1,\psi\>$ for all $\psi\in\mathcal D_0$ and hence $[[T]]_0=T1$.  We will also use a generalized notion of $BMO$ here to extend the cancellation conditions $T1,T^*1\in BMO$, which were used in the $T1$ theorem from \cite{DJ}.  Let $M\geq0$ be an integer and $F\in\mathcal D_M'/\mathcal P$, that is $\mathcal D_M'$ modulo polynomials.  We say that $F\in BMO_M$ if
\begin{align*}
\sum_{k\in\Z}2^{2Mk}|Q_kF(x)|^2dx\,\delta_{t=2^{-k}}
\end{align*}
is a Carleson measure for any $\psi\in\mathcal D_M$, where $Q_kf=\psi_k*f$ and $\psi_k(x)=2^{kn}\psi(2^kx)$.  This definition agrees with the classical definition of $BMO$.  That is, for $F\in BMO_0$,
\begin{align*}
\sum_{k\in\Z}|Q_kF(x)|^2dx\,\delta_{t=2^{-k}}
\end{align*}
is a Carleson measure, and hence $F\in BMO$ by the $BMO$ characterization in terms of Carleson measures in \cite{C,J}.  A similar polynomial growth $BMO_M$ was defined by Youssfi \cite{Y}.  We use this polynomial growth $BMO_M$ to quantify our cancellation conditions for operators $T\in CZO(M+\gamma)$ in the following result.


\begin{theorem}\label{t:reducedCZbound}
Let $T\in CZO(M+\gamma)$ be bounded on $L^2$ and define $L=\lfloor M/2\rfloor$ and $\delta=(M-2L+\gamma)/2$.  If $T^*(x^\alpha)=0$ in $\mathcal D_M'$ for all $|\alpha|\leq L$ and $[[T]]_\alpha\in BMO_{|\alpha|}$ for all $|\alpha|\leq L$, then $T$ extends to a bounded operator on $H^p$ for $\frac{n}{n+L+\delta}<p\leq1$.
\end{theorem}

Recall here that the operator $T^*$ is defined from $\S$ into $\S'$ via $\<T^*f,g\>=\<Tg,f\>$, and the definition of $T^*$ is extended to an operator from $\mathcal O_M$ to $\mathcal D_M'$ by the methods discussed above.  Note also that this is not a full necessary and sufficient theorem for Hardy space bounds as described above.  This theorem will be used to prove the boundedness of certain paraproduct operators, which in turn allow us to prove the full necessary and sufficient theorem, which is stated in Theorem \ref{t:CZbound} at the end of this section.

The choice of $L$ and $\delta$ here are such that $L\geq0$ is an integer, $0<\delta\leq1$, and $2(L+\delta)=M+\gamma$.  It is also not hard to see that $T^*(x^\alpha)=0$ for all $|\alpha|\leq L$ if and only if $[[T^*]]_\alpha=0$ for all $|\alpha|\leq L$.  We prove Theorem \ref{t:CZbound} by decomposing an operator $T\in CZO(M+\gamma)$ into a collection of operators $\{\Lambda_k\}\in LPSO(n+2L+2\delta,L+\delta')$ for $0<\delta'<\delta$ and applying Theorem \ref{t:sqbound}.  This decomposition of $T$ into a collection of Littlewood-Paley-Stein operators is stated precisely in the next theorem.

\begin{theorem}\label{t:CZtoLPS}
Let $T\in CZO(M+\gamma)$ for some integer $M\geq1$ and $0<\gamma\leq1$ be bounded on $L^2$, and fix $\psi\in\mathcal D_{M}$.  Also let $L=\lfloor M/2\rfloor$ and $\delta=(M-2L+\gamma)/2$.  If $T^*(x^\alpha)=0$ in $\mathcal D_{M}'$ for all $|\alpha|\leq L$, then $\{\Lambda_k\}\in LPSO(n+2L+2\delta,L+\delta')$ for all $0<\delta'<\delta$, where $\Lambda_k=Q_kT$ and $Q_kf(x)=\psi_k*f(x)$.  Furthermore, for $\frac{n}{n+L+\delta}<p\leq1$, $T$ extends to a bounded operator on $H^p$ if and only if $S_{\Lambda}$ extends to a bounded operator from $H^p$ into $L^p$.
\end{theorem}

Throughout, we write $L^p=L^p(\R^n)$ and $H^p=H^p(\R^n)$ for $0<p<\infty$.  We will also apply Theorem \ref{t:CZbound} to Bony paraproducts operator, which were originally defined in \cite{B} and famously applied in the $T1$ theorem \cite{DJ} (see also \cite{BMN}).  Let $\psi\in\mathcal D_{L+1}$ for some $L\geq0$ and $\varphi\in C_0^\infty$.  Define $Q_kf=\psi_k*f$ and $P_kf=\varphi_k*f$.  For $\beta\in BMO$, define
\begin{align}
\Pi_\beta f(x)&=\sum_{j\in\Z}Q_j\(Q_j\beta\cdot P_jf\)(x).\label{paraproduct}
\end{align}
It easily follows that $\Pi_\beta\in CZO(M+\gamma)$ for all $M\geq0$ and $0<\gamma\leq1$.  It is well known that $\Pi_\beta^*(1)=0$, and if one selects $\psi$ and $\varphi$ appropriately, it also follows that $\Pi_\beta(1)=\beta$ in $BMO$ as well.  We are not interested in an exact identification of $\Pi_\beta(1)$ in this work, so we don't worry about the extra conditions that should be imposed on $\psi$ and $\varphi$ to assure that $\Pi_\beta(1)=\beta$.  

\begin{theorem}\label{t:Bonyparaproduct}
Let $\Pi_\beta$ be as in \eqref{paraproduct} for $\beta\in BMO$, $\psi\in\mathcal D_{L+1}$, and $\varphi\in C_0^\infty$.  Then $\Pi_\beta$ is bounded on $H^p$ for all $\frac{n}{n+L+1}<p\leq1$.
\end{theorem}
By Theorem \ref{t:Bonyparaproduct} it is possible to construct $\Pi_\beta$ so that it is bounded on $H^p$ for $p>0$ arbitrarily small by choosing $\psi\in\mathcal D_{L+1}$ for $L$ sufficiently large.  It should be noted that some Hardy space estimates for a variant of the Bony paraproduct in \eqref{paraproduct} were proved in \cite{GK}.  Although we use a different construction of the paproduct, so we will prove Theorem \ref{t:Bonyparaproduct} here as well.  Finally, we state the first necessary and sufficient boundedness theorem for Calder\'on-Zygmund operators on Hardy spaces.

\begin{theorem}\label{t:CZbound}
Let $T\in CZO(M+\gamma)$ be bounded on $L^2$ and define $L=\lfloor M/2\rfloor$ and $\delta=(M-2L+\gamma)/2$.  Then $T^*(x^\alpha)=0$ in $\mathcal D_M'$ for all $|\alpha|\leq L$ if and only if $T$ extends to a bounded operator on $H^p$ for $\frac{n}{n+L+\delta}<p\leq1$.
\end{theorem}

Note that Theorem \ref{t:reducedCZbound} is made obsolete by Theorem \ref{t:CZbound}.  We state Theorem \ref{t:reducedCZbound} separately since we will use it to prove the stronger Theorem \ref{t:CZbound}.  More precisely, we will prove Theorem \ref{t:reducedCZbound}, apply Theorem \ref{t:reducedCZbound} to prove $H^p$ bounds for Bony paraproducts in Theorem \ref{t:Bonyparaproduct}, and finally we will prove Theorem \ref{t:CZbound} with the help of Theorem \ref{t:Bonyparaproduct} and a result from \cite{T,FHJW,FTW}.  In this way, Theorems \ref{t:reducedCZbound}, \ref{t:Bonyparaproduct}, and \ref{t:CZbound} are proved in that order, with each depending on the previous results.

The rest of the article is organized as follows.  In Section 2, we establish some notation and preliminary results.  Section 3 is dedicated to Littlewood-Paley-Stein square functions and proving Theorem \ref{t:sqbound}.  In section 4, we prove the singular integral operator results in Theorems \ref{t:reducedCZbound} and \ref{t:CZtoLPS}.  In section 5, we apply Theorem \ref{t:CZbound} to the Bony paraproducts to prove Theorem \ref{t:Bonyparaproduct}.  In the last section, we use Theorem \ref{t:Bonyparaproduct} and a result from \cite{T,FHJW,FTW} to prove Theorem \ref{t:CZbound}.

\section{Preliminaries}\label{s:preliminaries}

We use the notation $A\less B$ to mean that $A\leq CB$ for some constant $C$.  The constant $C$ is allowed to depend on the ambient dimension, smoothness and decay parameters of our operators, indices of function spaces etc.; in context, the dependence of the constants is clear.  Recall that we define $\Phi_k^N(x)=2^{kn}(1+2^k|x|)^{-N}$.  It is easy to verify that $\Phi_k^N(x)\leq\Phi_k^{\widetilde N}(x)$ for $\widetilde N\leq N$, and it is well known that
\begin{align*}
\Phi_j^N*\Phi_k^N(x)\less\Phi_{\min(j,k)}^N(x).
\end{align*}
We will use these inequalities many times throughout this work without specifically referring to them.

We will use the following Frazier and Jawerth type discrete Calder\'on reproducing formula \cite{FJ} (see also \cite{HL} for a multiparameter formulation of this reproducing formula): there exist $\phi_j,\tilde\phi_j\in\S$ for $j\in\Z$ with infinite vanishing moment such that
\begin{align}
f(x)=\sum_{j\in\Z}\sum_{\ell(Q)=2^{-(j+N_0)}}|Q|\,\phi_j(x-c_Q)\tilde\phi_j*f(c_Q)\text{ in }L^2\label{discretereproducingformula}
\end{align}
for $f\in L^2$.  The summation in $Q$ here is over all dyadic cubes with side length $\ell(Q)=2^{-(j+N_0)}$, where $N_0$ is some large constant, and $c_Q$ denotes the center of cube $Q$.  Throughout this paper, we reserve the notation $\phi_j$ and $\tilde\phi_j$ for the operators constructed in this discrete Calder\'on decomposition.

We will also use a more traditional formulation of Calder\'on's reproducing formula:  fix $\varphi\in C_0^\infty(B(0,1))$ with integral $1$ such that 
\begin{align}
\sum_{k\in\Z} Q_kf=f\text{ in }L^2\label{reproducingformula}
\end{align}
for $f\in L^2$, where $\psi(x)=2^n\varphi(2x)-\varphi(x)$, $\psi_k(x)=2^{kn}\psi(2^kx)$, and $Q_kf=\psi_k*f$.  Furthermore, we can assume that $\psi$ has an arbitrarily large, but fixed, number of vanishing moments.  Again we will reserve the notation $\psi_k$ and $Q_k$ for convolution operators with convolution kernels in $\mathcal D_M$ for some $M\geq0$.  For this work, the most important difference between the functions $\psi$ and $\phi$ is that $\psi$ is compactly supported, while $\phi$ is necessarily not compactly supported.  We will use formula \eqref{discretereproducingformula} to decompose square functions and formula \eqref{reproducingformula} to decompose Calder\'on-Zygmund operators.

There are many equivalent definitions of the real Hardy spaces $H^p=H^p(\R^n)$ for $0<p<\infty$.  We use the following one.  Define the non-tangential maximal function 
\begin{align*}
\mathcal N^\varphi f(x)=\sup_{t>0}\sup_{|x-y|\leq t}\left|\int_{\R^n}t^{-n}\varphi(t^{-1}(y-u))*f(u)du\right|,
\end{align*}
where $\varphi\in\S$ with non-zero integral.  It was proved by Fefferman and Stein in \cite{FS2} that one can define $||f||_{H^p}=||\mathcal N^\varphi f||_{L^p}$ to obtain the classical real Hardy spaces $H^p$ for $0<p<\infty$.  It was also proved in \cite{FS2} that for any $\varphi\in\S$ and $f\in H^p$ for $0<p<\infty$,
\begin{align*}
\left|\left|\sup_{k\in\Z}|\varphi_k*f|\right|\right|_{L^p}\less||f||_{H^p}.
\end{align*}
We will use a number of equivalent semi-norms for $H^p$.  Let $\psi\in\mathcal D_M$ for some integer $M> n(1/p-1)$, and let $\psi_k$ and $Q_k$ be as above, satisfying \eqref{reproducingformula}.  For $f\in \S'/\mathcal P$ (tempered distributions modulo polynomials), $f\in H^p$ if and only if 
\begin{align*}
\left|\left|\(\sum_{k\in\Z}|Q_kf|^2\)^\frac{1}{2}\right|\right|_{L^p}<\infty,
\end{align*}
and this quantity is comparable to $||f||_{H^p}$.  The space $H^p$ can also be characterized by the operators $\phi_j$ and $\tilde\phi_j$ from the discrete Littlewood-Paley-Stein decomposition in \eqref{discretereproducingformula}.  This characterization is given by the following, which can be found in \cite{HL,LZ}.  Given $0<p<\infty$
\begin{align*}
\left|\left|\(\sum_{j\in\Z}\sum_{\ell(Q)=2^{-(j+N_0)}}|\tilde\phi_j*f(c_Q)|^2\chi_Q\)^\frac{1}{2}\right|\right|_{L^p}\approx||f||_{H^p},
\end{align*}
where $\chi_E(x)=1$ for $x\in E$ and $\chi_E(x)=0$ for $x\notin E$ for a subset $E\subset\R^n$.  The summation again is indexed by all dyadic cubes $Q$ with side length $\ell(Q)=2^{-(j+N_0)}$  For a continuous function $f:\R^n\rightarrow\C$ and $0<r<\infty$, define
\begin{align}
&\mathcal M_{j}^{r}f(x)=\left\{\mathcal M\[\(\sum_{\ell(Q)=2^{-(j+N_0)}}f(c_Q)\chi_Q\)^r\;\](x)\right\}^\frac{1}{r},\label{defnMr}
\end{align}
where $\mathcal M$ is the Hardy-Littlewood maximal operator.  The following estimate was also proved in \cite{HL}.

\begin{proposition}\label{p:Hpmaximal}
For any $\nu>0$, $\frac{n}{n+\nu}<r<p\leq1$, and $f\in H^p$
\begin{align*}
\left|\left|\(\sum_{j\in\Z}\(\mathcal M_j^{r}(\tilde\phi_j*f)\)^2\)^\frac{1}{2}\right|\right|_{L^p}\less||f||_{H^p},
\end{align*}
where $\mathcal M_j^{r}$ is defined as in \eqref{defnMr}.
\end{proposition}


The next result is a rehash of an estimate proved in \cite{HL}; their estimate was in the multiparameter setting, whereas the one here is the single parameter version.

\begin{proposition}\label{p:discretebound}
Let $f:\R^n\rightarrow\C$ a non-negative continuous function, $\nu>0$, and $\frac{n}{n+\nu}<r\leq 1$.  Then 
\begin{align*}
&\sum_{\ell(Q)=2^{-(j+N_0)}}|Q|\,\Phi_{\min(j,k)}^{n+\nu}(x-c_Q)f(c_Q)\less2^{\max(0,j-k)\nu}\mathcal M_{j}^{r}f(x)
\end{align*}
for all $x\in\R^n$, where $\mathcal M_j^{r}$ is defined in \eqref{defnMr} and the summation indexed by $\ell(Q)=2^{-(j+N_0)}$ is the sum over all dyadic cubes with side length $2^{-(j+N_0)}$ and $c_Q$ denotes the center of cube $Q$.  
\end{proposition}

\begin{proof}
Define
\begin{align*}
A_0&=\{Q\text{ dyadic}:\ell(Q)=2^{-(j+N_0)}\text{ and }|x-c_Q|\leq2^{-(j+N_0)}\}\\
A_\ell&=\{Q\text{ dyadic}:\ell(Q)=2^{-(j+N_0)}\text{ and }2^{\ell-1-(j+N_0)}<|x-c_Q|\leq2^{\ell-(j+N_0)}\}
\end{align*}
for $\ell\geq1$.  Now for each $Q\in A_0$
\begin{align*}
\Phi_{\min(j,k)}^{n+\nu}(x-c_Q)&=\frac{2^{\min(j,k)n}}{(1+2^{\min(j,k)}|x-c_Q|)^{n+\nu}}\leq2^{\min(j,k)n}\leq2^{jn},
\end{align*}
and for each $Q\in A_\ell$ when $\ell\geq1$ 
\begin{align*}
\Phi_{\min(j,k)}^{n+\nu}(x-c_Q)&=\frac{2^{\min(j,k)n}}{(1+2^{\min(j,k)}|x-c_Q|)^{n+\nu}}\leq\frac{2^{\min(j,k)n}}{(1+2^{\min(j,k)}(2^{\ell-1-(j+N_0)}))^{n+\nu}}\\
&\leq2^{\min(j,k)n}2^{-(n+\nu)\min(j,k)}2^{-(n+\nu)\ell+n+\nu+(n+\nu)(j+N_0)}\\
&\less2^{\max(0,j-k)\nu}2^{-(n+\nu)\ell}2^{jn}
\end{align*}
Since $\bigcup_\ell A_\ell$ makes up the collection of all dyadic cubes with side length $2^{-(j+N_0)}$, it follows that
\begin{align*}
\sum_{\ell(Q)=2^{-(j+N_0)}}|Q|\,\Phi_{\min(j,k)}^{n+\nu}(x-c_Q)f(c_Q)&=\sum_{\ell=0}^\infty\sum_{Q\in A_\ell}2^{-(j+N_0)n}\Phi_{\min(j,k)}^{n+\nu}(x-c_Q)f(c_Q)\\
&\hspace{0cm}\less\sum_{Q\in A_0}f(c_Q)+2^{\max(0,j-k)\nu}\sum_{\ell=1}^\infty2^{-\ell (n+\nu)}\sum_{Q\in A_\ell}f(c_Q)\\
&\hspace{0cm}\leq2^{\max(0,j-k)\nu}\sum_{\ell=0}^\infty2^{-\ell (n+\nu)}\(\sum_{Q\in A_\ell}f(c_Q)^r\)^\frac{1}{r}.
\end{align*}
For $Q\in A_\ell$ and $y\in Q$ it follows that 
\begin{align*}
|x-y|\leq|x-c_Q|+|y-c_Q|\leq2^{-(j+N_0)}+2^{\ell-(j-N_0)}\leq2^{\ell+1-(j+N_0)},
\end{align*}
Hence $\bigcup_{Q\in A_\ell}Q\subset B(x,2^{\ell+1-(j+N_0)})$.  We also have that $|A_\ell|\geq 2^{n(\ell-2)}$; so   
\begin{align*}
\left|\bigcup_{Q\in A_\ell}Q\right|\geq2^{-(j+N_0)n}2^{n(\ell-2)}=2^{-2n}2^{(\ell-(j+N_0))n}\geq|B(0,1)|^{-1}2^{-2n}|B(0,2^{\ell-(j+N_0)}|.
\end{align*}
Now we estimate the sum in $Q$ above:
\begin{align*}
\sum_{Q\in A_\ell}f(c_Q)^r&\leq\frac{1}{|\bigcup_{Q\in A_\ell}Q|}\int_{\bigcup_{Q\in A_\ell}Q}\chi_{\bigcup_{Q\in A_\ell}Q}(y)\sum_{Q\in A_\ell}f(c_Q)^rdy\\
&\leq\frac{1}{|\bigcup_{Q\in A_\ell}Q|}\int_{\bigcup_{Q\in A_\ell}Q}2^{(\ell+1)n}\sum_{Q\in A_\ell}f(c_Q)^r\chi_Q(y)dy\\
&\less\frac{2^{\ell n}}{|B(x,2^{\ell+1-(j+N_0)})|}\int_{B(x,2^{\ell+1-(j+N_0)})}\sum_{Q\in A_\ell}f(c_Q)^r\chi_{Q}(y)dy\\
&=\frac{2^{\ell n}}{|B(x,2^{\ell+1-(j+N_0)})|}\int_{B(x,2^{\ell+1-(j+N_0)})}\(\sum_{Q\in A_\ell}f(c_Q)\chi_{Q}(y)\)^rdy\\
&\less2^{\ell n}\mathcal M\[\(\sum_{Q\in A_\ell}f(c_Q)\chi_Q\)^r\;\](x).
\end{align*}
Then we have that
\begin{align*}
&\sum_{\ell(Q)=2^{-(j+N_0)}}|Q|\,\Phi_{\min(j,k)}^{n+\nu}(x-c_Q)f(c_Q)\\
&\hspace{2cm}\less2^{\max(0,j-k)\nu}\sum_{\ell=0}^\infty2^{-\ell (n+\nu-n/r)}\left\{\mathcal M\[\(\sum_{Q\in A_\ell}f(c_Q)\chi_Q\)^r\;\](x)\right\}^\frac{1}{r}\\
&\hspace{2cm}\less2^{\max(0,j-k)\nu}\left\{\mathcal M\[\(\sum_{\ell(Q)=2^{-(j+N)}}f(c_Q)\chi_Q\)^r\;\](x)\right\}^\frac{1}{r}.
\end{align*}
\end{proof}

We will also need some Carleson measure estimates for the result in Theorem \ref{t:sqbound}.  The next proof is a well known argument that can be found in \cite{C,J}.

\begin{proposition}\label{p:HpCarleson}
Suppose
\begin{align}
d\mu(x,t)=\sum_{k\in\Z}\mu_k(x)\delta_{t=2^{-k}}\,dx\label{discreteCarleson}
\end{align}
is a Carleson measure, where $\mu_k$ is a non-negative, locally integrable function for all $k\in\Z$.  Also let $\varphi\in\S$, and define $P_kf=\varphi_k*f$, where $\varphi_k(x)=2^{kn}\varphi(2^kx)$ for $k\in\Z$.  Then 
\begin{align*}
\left|\left|\(\sum_{k\in\Z}|P_kf|^p\mu_k\)^\frac{1}{p}\right|\right|_{L^p}\less||f||_{H^p}\hspace{.5cm}\text{ for all $0<p<\infty$}
\end{align*}
and
\begin{align*}
\left|\left|\(\sum_{k\in\Z}|P_kf|^2\mu_k\)^\frac{1}{2}\right|\right|_{L^p}\less||f||_{H^p}\hspace{.5cm}\text{ for all $0<p\leq2$}.
\end{align*}
\end{proposition}

\begin{proof}
Let $f\in H^p$, and we begin the proof of the the first estimate above by looking at
\begin{align*}
&\int_{\R^n}\sum_{k\in\Z}|P_kf(x)|^p\mu_k(x)dx\\
&\hspace{2cm}=p\int_0^\infty d\mu\(\left\{(x,t):\left|\int_{\R^n}t^{-n}\varphi(t^{-1}(x-y))f(y)dy\right|>\lambda\right\}\)\lambda^p\frac{d\lambda}{\lambda}.
\end{align*}
Define $E_\lambda=\{x:|\mathcal N^\varphi f(x)|>\lambda\}$, and it follows that 
\begin{align*}
\left\{(x,t):\left|\int_{\R^n}t^{-n}\varphi(t^{-1}(x-y))f(y)dy\right|>\lambda\right\}\subset\widehat E_\lambda,
\end{align*}
where $\widehat E=\{(x,t):B(x,t)\subset E\}$.  Therefore
\begin{align*}
\int_{\R^n}\sum_{k\in\Z}|P_kf(x)|^p\mu_k(x)dx&\leq p\int_0^\infty d\mu(\widehat E_\lambda)\lambda^p\frac{d\lambda}{\lambda}\\
&\less p\int_0^\infty | E_\lambda|\lambda^p\frac{d\lambda}{\lambda}=||\mathcal N^\varphi f||_{L^p}^p=||f||_{H^p}^p.
\end{align*}
Here we use that $d\mu(\widehat E)\less|E|$ for any open set $E\subset\R^n$, which is a well known estimate for Carleson measures.  In the case $p=2$, the second estimate coincides with the first and hence there is no more to prove.  When $0<p<2$, we set $r=\frac{2}{p}>1$ and then the H\"older conjugate of $r$ is $r'=\frac{2}{2-p}$.  Now applying the first estimate above, we finish the proof.
\begin{align*}
\int_{\R^n}\(\sum_{k\in\Z}|P_kf(x)|^2\mu_k(x)\)^\frac{p}{2}dx&\leq\int_{\R^n}\sup_k|P_kf(x)|^{(2-p)p/2}\(\sum_{k\in\Z}|P_kf(x)|^p\mu_k(x)\)^\frac{p}{2}dx\\
&\leq\left|\left|\(\mathcal N^\varphi f\)^{(2-p)p/2}\right|\right|_{L^{r'}}\left|\left|\(\sum_{k\in\Z}|P_kf(x)|^p\mu_k(x)\)^\frac{p}{2}\right|\right|_{L^r}\\
&=\left|\left|\mathcal N^\varphi f\right|\right|_{L^p}^\frac{p(2-p)}{2}\(\int_{\R^n}\sum_{k\in\Z}|P_kf(x)|^p\mu_k(x)dx\)^\frac{p}{2}\\
&\less||f||_{H^p}^{p(2-p)/2}||f||_{H^p}^{p^2/2}=||f||_{H^p}^p.
\end{align*}

\end{proof}

\section{Hardy Space Estimates for Square Functions}\label{s:sqbound0}

In this section we prove Theorem \ref{t:sqbound}.  To do this, we first prove a reduced version of the theorem.

\begin{lemma}\label{l:linearsqbound}
Assume $\{\Lambda_k\}\in LPSO(n+2L+2\delta,L+\delta)$ for some integer $L\geq0$ and $0<\delta\leq1$.  If $\Lambda_k(y^\alpha)=0$ for all $k\in\Z$ and $|\alpha|\leq L$, then $||S_\Lambda f||_{L^p}\less||f||_{H^p}$ for all $f\in H^p\cap L^2$ and $\frac{n}{n+L+\delta}<p\leq1$.
\end{lemma}

We call this a reduced version of Theorem \ref{t:sqbound} because we have strengthened the assumptions of from the Carleson measure estimates for \eqref{dmuCarleson} to the vanishing moment type assumption above; $\Lambda_k(y^\alpha)=0$ for $|\alpha|\leq L$.

\begin{proof}
Fix $\nu\in(n/p-n,L+\delta)$, which is possible since our assumption on $p$ implies that $\frac{n}{p}-n<L+\delta$.  Also fix $r\in(0,1)$ such that $\frac{n}{n+\nu}<r<p$.  Let $f\in H^{p}\cap L^2$, and we decompose
\begin{align*}
\Lambda_kf(x)&=\sum_{j\in\Z}\sum_{Q}|Q|\tilde\phi_{j}*f(c_Q)\Lambda_k\psi_{j}^{c_Q}(x)=\sum_{j\in\Z}\sum_{Q}|Q|\tilde\phi_{j}*f(c_{Q})\int_{\R^{n}}\lambda_k(x,y)\psi_{j}^{c_{Q}}(y)dy.
\end{align*}
The summation in $Q$ is over all dyadic cubes with side lengths $\ell(Q)=2^{-(j+N_0)}$.  Then we have the following almost orthogonality estimates
\begin{align*}
\left|\int_{\R^{n}}\lambda_k(x,y)\phi_{j}^{c_{Q}}(y)dy\right|&=\left|\int_{\R^{n}}\lambda_k(x,y)\(\phi_{j}^{c_{Q}}(y)-\sum_{|\alpha|\leq L}\frac{D^\alpha\phi_{j}^{c_{Q}}(x)}{\alpha!}(y-x)^\alpha\)dy\right|\\
&\hspace{-2cm}\less\int_{\R^{n}}\Phi_k^{n+2L+2\delta}(x-y)(2^{j}|x-y|)^{L+\delta}\(\Phi_{j}^{n+L+\delta}(y-c_{Q})+\Phi_{j}^{n+L+\delta}(x-c_{Q})\)dy\\
&\hspace{-2cm}\less2^{(L+\delta)(j-k)}\int_{\R^{n}}\Phi_k^{n+L+\delta}(x-y)\(\Phi_{j}^{n+L+\delta}(y-c_{Q})+\Phi_{j}^{n+L+\delta}(x-c_{Q})\)dy\\
&\hspace{-2cm}\less2^{(L+\delta)(j-k)}\Phi_{\min(j,k)}^{n+L+\delta}(x-c_{Q}).
\end{align*}
Also, using the vanishing moment properties of $\phi_{j}$, we have the following estimate,
\begin{align*}
\left|\int_{\R^{n}}\lambda_k(x,y)\phi_{j}^{c_{Q}}(y)dy\right|&=\left|\int_{\R^{n}}\(\lambda_k(x,y)-\sum_{|\alpha|\leq L}\frac{D_1^\alpha\lambda_k(x,c_{Q})}{\alpha!}(x-y)^\alpha\)\phi_{j}^{c_{Q}}(y)dy\right|\\
&\less\int_{\R^{n}}\Phi_{k}^{n+L+\delta}(x-y)(2^k|y-c_{Q}|)^{L+\delta}\Phi_{j}^{n+2L+2\delta}(y-c_{Q})dy\\
&\hspace{1.5cm}+\int_{\R^{n}}\Phi_{k}^{n+L+\delta}(x-c_{Q})(2^k|y-c_{Q}|)^{L+\delta}\Phi_{j}^{n+2L+2\delta}(y-c_{Q})dy\\
&\less2^{(L+\delta)(k-j)}\int_{\R^{n}}\Phi_{k}^{n+L+\delta}(x-y)\Phi_{j}^{n+L+\delta}(y-c_{Q})dy\\
&\hspace{1.5cm}+2^{(L+\delta)(k-j)}\int_{\R^{n}}\Phi_{k}^{n+L+\delta}(x-c_{Q})\Phi_{j}^{n+L+\delta}(y-c_{Q})dy\\
&\hspace{.1cm}\less2^{(L+\delta)(k-j)}\Phi_{\min(j,k)}^{n+L+\delta}(x-c_{Q}).
\end{align*}
Therefore
\begin{align*}
&\left|\int_{\R^{n}}\lambda_k(x,y)\phi_{j}^{c_{Q}}(y)dy\right|\less2^{-(L+\delta)|j-k|}\Phi_{\min(j,k)}^{n+\nu}(x-c_{Q}).
\end{align*}
Applying Proposition \ref{p:discretebound} yields
\begin{align*}
|\Lambda_kf(x)|&\less\sum_{j\in\Z}\sum_{Q}|Q|\tilde\phi_{j}*f(c_{Q})2^{-(L+\delta)|j-k|}\Phi_{\min(j,k)}^{n+\nu}(x-c_{Q})\\
&\less\sum_{j\in\Z}2^{-(L+\delta)|j-k|}2^{\nu\max(0,k-j)}\mathcal M_{j}^{r}(\tilde\phi_j*f)(x)\leq\sum_{j\in\Z}2^{-\epsilon|j-k|}\mathcal M_{j}^{r}(\tilde\phi_j*f)(x),
\end{align*}
where $\epsilon=L+\delta-\nu>0$; recall that these parameter are chosen such that $\nu<L+\delta$.  Applying Proposition \ref{p:Hpmaximal} to $\mathcal M_{j}^{r}(\tilde\phi_j*f)$ (recall that $r$ was chosen such that $\frac{n}{n+\nu}<r<p$) yields the appropriate estimate below,
\begin{align*}
||S_\Lambda f||_{L^p}&\less\left|\left|\(\sum_{k\in\Z}\[\sum_{j\in\Z}2^{-\epsilon|j-k|}\mathcal M_{j}^{r}(\tilde\phi_j*f)\]^2\)^\frac{1}{2}\right|\right|_{L^p}\\
&\less\left|\left|\(\sum_{j,k\in\Z}2^{-\epsilon|j-k|}\[\mathcal M_{j}^{r}(\tilde\phi_j*f)\]^2\)^\frac{1}{2}\right|\right|_{L^p}\less||f||_{H^p}.
\end{align*}
This completes the proof of Lemma \ref{l:linearsqbound}.
\end{proof}

Next we construct paraproducts to decompose $\Lambda_k$.  Fix an approximation to identity operator $P_kf=\varphi_k*f$, where $\varphi_k(x)=2^{kn}\varphi(2^kx)$ and $\varphi\in\S$ with integral $1$.  Define for $\alpha,\beta\in\N_0^n$ 
\begin{align*}
M_{\alpha,\beta}&=\left\{
\begin{array}{ll}
\ds{(-1)^{|\beta|-|\alpha|}\frac{\beta!}{(\beta-\alpha)!}\int_{\R^n}\varphi(y)y^{\beta-\alpha}dy}	&\alpha\leq \beta\\
\\
0																			&\alpha\not\leq\beta
\end{array}
\right..
\end{align*}
Here we say $\alpha\leq\beta$ for $\alpha=(\alpha_1,...,\alpha_n),\beta=(\beta_1,...,\beta_n)\in\N_0^n$ if $\alpha_i\leq\beta_i$ for all $i=1,...,n$.  It is clear that $|M_{\alpha,\beta}|<\infty$ for all $\alpha,\beta\in\N_0^n$ since $\varphi\in\S$.  Also note that when $|\alpha|=|\beta|$
\begin{align}
M_{\alpha,\beta}=\left\{\begin{array}{ll}
\beta!		&\alpha=\beta\\
0			&\alpha\neq\beta\text{ and }|\alpha|=|\beta|
\end{array}
\right..\label{Lalphabeta}
\end{align}
We consider the operators $P_kD^\alpha$ defined on $\S'$, where $D^\alpha$ is taken to get the distributional derivative acting on $\S'$.  Hence $P_kD^\alpha f(x)$ is well defined for $f\in\S'$ since $P_kD^\alpha f(x)=\<\varphi_k^x,D^\alpha f\>=(-1)^{|\alpha|}\<D^\alpha(\varphi_k^x),f\>$ and $D^\alpha(\varphi_k^x)\in \S$.  In fact, this gives a kernel representation for $P_kD^\alpha$; estimates for this kernel are addressed in the proof of Proposition \ref{p:paraproducts}.  We also have
\begin{align*}
[[P_kD^\alpha]]_{\beta}(x)&=2^{|\beta|k}\int_{\R^n}\varphi_k(x-y)\partial_y^\alpha((x-y)^\beta)dy
=2^{k|\alpha|}M_{\alpha,\beta}.
\end{align*}
For $k\in\Z$, define
\begin{align}
\Lambda_k^{(0)}f(x)&=\Lambda_kf(x)-[[\Lambda_k]]_0(x)\cdot P_kf(x),\text{ and }\label{defnLambda0}\\
\Lambda_k^{(m)}f(x)&=\Lambda_k^{(m-1)}f(x)-\sum_{|\alpha|=m}(-1)^{|\alpha|}\frac{[[\Lambda_k^{(m-1)}]]_\alpha(x)}{\alpha!}\cdot 2^{-k|\alpha|}P_kD^\alpha f(x).\label{defnLambdam}
\end{align}
for $1\leq m\leq L$.  

\begin{proposition}\label{p:paraproducts}
Let $\{\Lambda_k\}\in LPSO(N,L+\delta)$, where $N=n+2L+2\delta$ for some integer $L\geq0$ and $0<\delta\leq1$, and assume that 
\begin{align}
d\mu_\alpha(x,t)=\sum_{k\in\Z}|[[\Lambda_k]]_{\alpha}(x)|^2\delta_{t=2^{-k}}\,dx\label{linearmomentcondition}
\end{align}
is a Carleson measure for all $\alpha\in\N_0^n$ such that $|\alpha|\leq L$.  Also let $\Lambda_k^{(m)}$ be as in as in \eqref{defnLambda0} and \eqref{defnLambdam} for $0\leq m\leq L$.  Then $\Lambda_k^{(m)}\in LPSO(N,L+\delta)$ for the same $N$, $L$, and $\delta$, and satisfy the following:
\begin{itemize}

\item[(1)] $[[\Lambda_k^{(m)}]]_\alpha=0$ for all $\alpha\in\N_0^n$ with $|\alpha|\leq m\leq L$.

\medskip

\item[(2)] $d\mu_m(x,t)$ is a Carleson measure for all $0\leq m\leq L$, where $d\mu_m$ is defined
\begin{align*}
d\mu_m(x,t)=\sum_{k\in\Z}\sum_{|\alpha|\leq L}|[[\Lambda_k^{(m)}]]_\alpha(x)|^2\delta_{t=2^{-k}}\,dx.
\end{align*}
\end{itemize}
\end{proposition}

\begin{proof}
Since $\{\Lambda_k\}\in LPSO(n+2L+2\delta,L+\delta)$, we know that $|[[\Lambda_k]]_\alpha(x)|\less1$ for all $|\alpha|\leq L$.  Then to verify that $\{\Lambda_k^{(m)}\}\in LPSO(n+2L+2\delta,L+\delta)$ for $0\leq m\leq L$, it is sufficient to show that $\{2^{-k|\alpha|}P_kD^\alpha\}\in LPSO(n+2L+2\delta,L+\delta)$ for all $\alpha\in\N_0^n$.  For $f\in \S'$, we have the following integral representation for $2^{-k|\alpha|}P_kD^\alpha f$, which was alluded to above,
\begin{align*}
2^{-k|\alpha|}P_kD^\alpha f(x)=(-1)^{|\alpha|}2^{-k|\alpha|}\<D^\alpha(\varphi_k^x),f\>&
=(-1)^{|\alpha|}(D^\alpha\varphi)_k*f(x).
\end{align*}
Since $\varphi\in \S$, it easily follows that $D^\alpha\varphi\in \S$ for all $\alpha\in\N_0^n$ and that $\{2^{-k|\alpha|}P_kD^\alpha\}\in LPSO(n+2L+2\delta,L+\delta)$.  Now we prove (1) by induction:  the $m=0$ case for (1) is not hard to verify
\begin{align*}
[[\Lambda_k^{(0)}]]_0=\Lambda_k1-[[\Lambda_k]]_0\cdot P_k1=[[\Lambda_k]]_0-[[\Lambda_k]]_0=0.
\end{align*}
Now assume that (1) holds for $m-1$, that is, assume $[[\Lambda_k^{(m-1)}]]_\alpha=0$ for all $|\alpha|\leq m-1$.  Then for $|\beta|\leq m-1$
\begin{align*}
[[\Lambda_k^{(m)}]]_\beta &=[[\Lambda_k^{(m-1)}]]_\beta -\sum_{|\alpha|=m}\frac{[[\Lambda_k^{(m-1)}]]_\alpha}{\alpha!}(-1)^{|\alpha|} M_{\alpha,\beta}=0.
\end{align*}
The first term here vanished by the inductive hypothesis.  The second term is zero since $|\beta|<m=|\alpha|$ and hence $M_{\alpha,\beta}=0$.  For $|\beta|=m$,
\begin{align*}
[[\Lambda_k^{(m)}]]_\beta &=[[\Lambda_k^{(m-1)}]]_\beta -\sum_{|\alpha|=m}\frac{[[\Lambda_k^{(m-1)}]]_\alpha}{\alpha!}(-1)^{|\alpha|} M_{\alpha,\beta}=[[\Lambda_k^{(m-1)}]]_\beta -[[\Lambda_k^{(m-1)}]]_\beta =0,
\end{align*}
where the sum collapses using \eqref{Lalphabeta}.  By induction, this verifies (1) for all $m\leq L$.  Given the Carleson measure assumption for $d\mu_\alpha(x,t)$ in \eqref{linearmomentcondition}, one can easily prove (2) if the following statement holds:  for all $0\leq m\leq L$
\begin{align}
\sum_{|\alpha|\leq L}|[[\Lambda_k^{(m)}]]_\alpha(x)|&\leq (1+C_0)^{m+1}\sum_{|\alpha|\leq L}|[[\Lambda_k]]_\alpha(x)|,\text{ where }C_0=\sum_{|\alpha|,|\beta|\leq L}|M_{\alpha,\beta}|.\label{induction}
\end{align}
We verify \eqref{induction} by induction.  For $m=0$, let $|\beta|\leq L$, and it follows that
\begin{align*}
[[\Lambda_k^{(0)}]]_\beta &=[[\Lambda_k]]_\beta  -[[\Lambda_k]]_0 \cdot [[P_k]]_\beta =[[\Lambda_k]]_\beta  -[[\Lambda_k]]_0 \cdot M_{0,\beta}
\end{align*}
Then
\begin{align*}
\sum_{|\beta|\leq L}|[[\Lambda_k^{(0)}]]_\beta |&\leq\sum_{|\beta|\leq L}|[[\Lambda_k]]_\beta  |+\sum_{|\beta|\leq L}|[[\Lambda_k]]_0 ||M_{0,\beta}|\leq(1+C_0)\sum_{|\beta|\leq L}|[[\Lambda_k]]_\beta  |.
\end{align*}
Now assume that \eqref{induction} holds for $m-1$, and consider
\begin{align*}
\sum_{|\beta|\leq L}|[[\Lambda_k^{(m)}]]_\beta |&\leq\sum_{|\beta|\leq L}|[[\Lambda_k^{(m-1)}]]_\beta |+\sum_{|\beta|\leq L}\sum_{|\alpha|=m}|[[\Lambda_k^{(m-1)}]]_\alpha ||M_{\alpha,\beta}|\\
&\leq\(1+\sum_{|\alpha|\leq m,|\beta|\leq L}|M_{\alpha,\beta}|\)\sum_{|\beta|\leq L}|[[\Lambda_k^{(m-1)}]]_\beta |\\
&\leq(1+C_0)\sum_{|\beta|\leq L}|[[\Lambda_k^{(m-1)}]]_\beta |\leq(1+C_0)^{m+1}\sum_{|\beta|\leq L}|[[\Lambda_k]]_\beta |.
\end{align*}
We use the inductive hypothesis in the last inequality here to bound the $[[\Lambda^{(m-1)}]]_\beta$.  Then by induction, the estimate in \eqref{induction} holds for all $0\leq m\leq L$, and completes the proof.
\end{proof}

Now we use Lemma \ref{l:linearsqbound} and the paraproduct operators $\Lambda_k^{(m)}$ along with Propositions \ref{p:HpCarleson} and \ref{p:paraproducts} to prove Theorem \ref{t:sqbound}.

\begin{proof}[Proof of Theorem \ref{t:sqbound}]
By density, it is sufficient to prove that $||S_\Lambda f||_{L^p}\less||f||_{H^p}$ for $f\in H^p\cap L^2$.  We bound $\Lambda_k$ in the following way using the definitions of $\Lambda_k^{(m)}$ in \eqref{defnLambda0} and \eqref{defnLambdam};
\begin{align*}
|\Lambda_k(x)f|&\leq|\Lambda_k1(x)\cdot P_kf(x)|+|\Lambda_k^{(0)}f(x)|\\
&\leq|\Lambda_k1(x)\cdot P_kf(x)|+|\Lambda_k^{(1)}f(x)|+\sum_{|\alpha|=1}|[[\Lambda_k^{(0)}]]_{\alpha}(x)|\,2^{-k|\alpha|}|P_kD^\alpha f(x)|\\
&\leq|\Lambda_k1(x)\cdot P_kf(x)|+|\Lambda_k^{(L)}f(x)|+\sum_{m=1}^L\sum_{|\alpha|=m}|[[\Lambda_k^{(m-1)}]]_{\alpha}(x)|\,2^{-k|\alpha|}|P_kD^\alpha f(x)|.
\end{align*}
By Propositions \ref{p:HpCarleson} and \ref{p:paraproducts}, it follows that
\begin{align*}
\left|\left|\(\sum_{k\in\Z}|\Lambda_k1 P_kf|^2\)^\frac{1}{2}\right|\right|_{L^p}+\sum_{m=1}^L\sum_{|\alpha|=m}\left|\left|\(\sum_{k\in\Z}|[[\Lambda_k^{(m-1)}]]_{\alpha}2^{-|\alpha|k} P_kD^\alpha f|^2\)^\frac{1}{2}\right|\right|_{L^p}\less||f||_{H^p}.
\end{align*}
Also by Lemma \ref{l:linearsqbound}, it follows that
\begin{align*}
\left|\left|\(\sum_{k\in\Z}|\Lambda_k^{(L)}f|^2\)^\frac{1}{2}\right|\right|_{L^p}\less||f||_{H^p}.
\end{align*}
Therefore $S_\Lambda$ can be extended to a bounded operator from $H^p$ into $L^p$.
\end{proof}

\section{Hardy Space Bounds for Singular Integral Operators}\label{s:BCZO}

In this section, we prove Theorem \ref{t:reducedCZbound}.  This is a reduced version of Theorem \ref{t:CZbound} in the sense that we have strengthened the assumptions on $T$, and hence obtain only a sufficient condition, not necessary.  We will apply Theorem \ref{t:sqbound} to prove Theorem \ref{t:reducedCZbound}.  In order to do so, we prove the decomposition result in Theorem \ref{t:CZtoLPS}.

\begin{proof}[Proof of Theorem \ref{t:CZtoLPS}]
Let $\psi\in\mathcal D_M$.  It is not hard to check that $T^*\psi_k^x(y)$ is the kernel of $Q_kT$, where $\psi_k^x(y)=\psi_k(y-x)$.  Also let $L=\lfloor M/2\rfloor$ and $\delta=(M-2L+\gamma)/2$.  We first verify \eqref{sqker1}-\eqref{sqker3} for $|x-y|>2^{3-k}$.  Assume that $|x-y|>2^{3-k}$.  Then for $|\alpha|\leq L$
\begin{align*}
|\partial_y^\alpha T^*\psi^x(y)|&=\left|\partial_y^\alpha\int_{\R^n}\(K(u,y)-\sum_{|\beta|\leq M}\frac{D_0^\beta K(x,y)}{\beta!}(u-x)^\beta\)\psi_k(u-x)du\right|\\
&=\left|\int_{\R^n}\(D_1^\alpha K(u,y)-\sum_{|\beta|\leq M}\frac{D_0^\beta D_1^\alpha K(x,y)}{\beta!}(u-x)^\beta\)\psi_k(u-x)du\right|\\
&\less\int_{\R^n}\frac{|x-u|^{M+\gamma}}{|x-y|^{n+|\alpha|+M+\gamma}}|\psi_k(u-x)|du\\
&\less\frac{2^{-k(M+\gamma)}}{(2^{-k}+|x-y|)^{n+|\alpha|+M+\gamma}}\int_{\R^n}|\psi_k(u-x)|du\\
&\less2^{|\alpha|k}\Phi_k^{n+M+|\alpha|+\gamma}(x-y)\leq2^{|\alpha|k}\Phi_k^{n+2L+2\delta}(x-y).
\end{align*}
If $M\geq1$, then this estimate holds for all $|\alpha|\leq L+1$.  In this case, the above estimate implies that \eqref{sqker3} also holds for $N=n+2L+2\delta$ and any $0<\delta\leq1$.  So it remains to verify \eqref{sqker3} for $M=1$, in which case $L=0$ and $\delta=\gamma/2$.  If $|y-y'|\geq2^{-k}$, then property \eqref{sqker3} easily follows from the estimate just proved with $\alpha=0$.  Otherwise we assume that $|y-y'|<2^{-k}$, and it follows that $|x-y'|\geq|x-y|-|y-y'|>|x-y|/2\geq2^{1-k}$.  Then
\begin{align*}
|T^*\psi^x(y)-T^*\psi^x(y')|&=\left|\int_{\R^n}\(K(u,y)- K(u,y')\)\psi_k(u-x)du\right|\\
&\hspace{0cm}=\left|\int_{\R^n}\( \(K(u,y)- K(u,y')\)-(K(x,y)- K(x,y'))\)\psi_k(u-x)du\right|\\
&\hspace{0cm}\leq\int_{\R^n}\sum_{|\beta|= 1}|D_0^\beta  K(\xi,y)-D_0^\beta  K(\xi,y')|\,|u-x|\,|\psi_k(u-x)|du\\
&\hspace{3cm}\text{ for some }\xi=cx+(1-c)u\text{ with }0<c<1\\
&\hspace{0cm}\less\int_{\R^n}\frac{|y-y'|^{\gamma}|x-u|}{|\xi-y|^{n+1+\gamma}}|\psi_k(u-x)|du\\
&\hspace{0cm}\less\frac{|y-y'|^{\gamma}2^{-k}}{(2^{-k}+|x-y|)^{n+1+\gamma}}=2^{\delta k}|y-y'|^{\delta}\Phi_k^{n+2L+2\delta}(x-y).
\end{align*}
Recall this is the situation where $M=1$, $L=0$, $\delta=\gamma/2$, and $|y-y'|\leq2^{-k}$, and hence in the last line $n+\gamma=n+2L+2\delta$ and $2^{\gamma k}|y-y'|^\gamma\leq 2^{\delta k}|y-y'|^\delta$.  This completes the proof of \eqref{sqker1}-\eqref{sqker3} for $|x-y|>2^{3-k}$.

When $|x-y|\leq 2^{3-k}$, we decompose $Q_kT$ further.  Let $\varphi\in C_0^\infty$ with integral $1$ such that $\widetilde\psi(x)=2^{n}\varphi(2x)-\varphi(x)$ and $\widetilde\psi\in\mathcal D_M$.  Then
\begin{align}
T^*\psi_k^x(y)&=\lim_{N\rightarrow\infty}P_NT^*\psi_k^x(y)=\sum_{\ell=k}^\infty \widetilde Q_\ell T^*\psi_k^x(y)+P_kT^*\psi_k^x(y).\label{furtherdecomp}
\end{align}
This equality holds pointwise almost everywhere since $T$ is a continuous operator from $L^2$ to $L^2$ and $\psi_k^x\in\mathcal D_M$.  Note that $\widetilde\psi,\psi\in\mathcal D_M$, and it is only this property that will be used throughout the rest of this proof.  So we abuse notation to make this proof a bit easier to read.  For the remainder of the proof, we will simply write $\widetilde \psi_\ell=\psi_\ell$ and $\widetilde Q_\ell=Q_\ell$ with the understanding that these two can actually be allowed to be different elements of $\mathcal D_M$.  Let $\alpha\in\N_0^n$ with $|\alpha|\leq L$.  Using the hypothesis $T^*(x^\mu)=0$ for $|\mu|\leq L$ we write
\begin{align*}
&\left|D_1^\alpha\<T\psi_{\ell}^{y},\psi_k^x\>\right|\leq |A_{\ell,k}(x,y)|+|B_{\ell,k}(x,y)|,\text{ where}\\
&\hspace{.25cm}A_{\ell,k}(x,y)=2^{\ell|\alpha|}\int_{|u-y|\leq2^{1-\ell}}T((D^\alpha\psi)_{\ell}^{y})(u)\(\psi_k^x(u)-\sum_{|\alpha|\leq L}\frac{D^\alpha\psi_k^x(y_1)}{\alpha!}(u-y)^\alpha\)du,\\
&\hspace{.25cm}B_{\ell,k}(x,y)=2^{\ell |\alpha|}\int_{|u-y|>2^{1-\ell}}T((D^\alpha\psi)_{\ell}^{y})(u)\(\psi_k^x(u)-\sum_{|\alpha|\leq L}\frac{D^\alpha\psi_k^x(y)}{\alpha!}(u-y)^\alpha\)du.
\end{align*}
The $A_{\ell,k}$ term is bounded as follows,
\begin{align*}
|A_{\ell,k}(x,y)|&\leq2^{\ell |\alpha|}||T((D^\alpha\psi)_{\ell}^{y})\cdot\chi_{B(y,2^{1-\ell})}||_{L^1}\\
&\hspace{2cm}\times\left|\left|\(\psi_k^x(u)-\sum_{|\alpha|\leq L}\frac{D^\alpha\psi_k^x(y)}{\alpha!}(\cdot-y)^\alpha\)\cdot\chi_{B(y,2^{1-\ell})}\right|\right|_{L^\infty}\\
&\hspace{0cm}\less2^{\ell |\alpha|}2^{-\ell n/2}||T((D^\alpha\psi)_{\ell}^{y})||_{L^2}2^{(L+\delta)(k-\ell)}2^{kn}\\
&\hspace{0cm}\less2^{\ell |\alpha|}2^{-\ell n/2}||(D^\alpha\psi)_{\ell}^{y}||_{L^2}2^{(L+\delta)(k-\ell)}2^{kn}\less2^{k|\alpha|}2^{\delta(k-\ell)}\Phi_k^{n+2L+2\delta}(x-y).
\end{align*}
Let $0<\delta\,'<\delta\,''<\delta$.  The $B_{\ell,k}$ term is bounded using the kernel representation of $T$
\begin{align*}
|B_{\ell,k}(x,y)|&\leq2^{\ell|\alpha|}\int_{|u-y|>2^{1-\ell}}\int_{\R^n}\left|K(u,v)-\sum_{|\beta|\leq L}\frac{D_1^\beta K(u,y)}{\beta!}(v-y)^\beta\right||(D^\alpha\psi)_\ell^y(u) |dv\\
&\hspace{4.5cm}\times\left|\psi_k^x(u)-\sum_{|\mu|\leq L}\frac{D^\mu(\psi_k^x)(y)}{\mu!}(u-y)^\mu\right|du
\end{align*}
\begin{align*}
&\hspace{-1.2cm}\less2^{\ell |\alpha|}\sum_{m=1}^\infty\int_{2^{m-\ell}<|u-y|\leq 2^{m+1-\ell}}\int_{\R^{n}}\frac{|v-y|^{L+\delta}}{|u-y|^{n+L+\delta}}|(D^\alpha\psi)_{\ell}^{y}(v)|dv2^{kn}(2^k|u-y|)^{L+\delta''}du\\
&\hspace{-1.2cm}\less2^{\ell |\alpha|}\sum_{m=1}^\infty\int_{2^{m-\ell}<|u-y|\leq 2^{m+1-\ell}}\int_{\R^{n}}\frac{2^{-(L+\delta)\ell}}{2^{(n+L+\delta)(m-\ell)}}|(D^\alpha\psi)_{\ell}^{y}(v)|dv 2^{kn}(2^k2^{m-\ell})^{L+\delta''}du\\
&\hspace{-1.2cm}\less2^{\ell |\alpha|}\sum_{m=1}^\infty 2^{(m-\ell)n}2^{-(L+\delta)\ell}2^{-(n+L+\delta)(m-\ell)}2^{kn}2^{(L+\delta\,'')(k+m-\ell)}\\
&\hspace{-1.3cm}\less2^{k |\alpha|}2^{(L-|\alpha|+\delta\,'')(k-\ell)}2^{kn}\sum_{m=1}^\infty 2^{(\delta\,''-\delta) m}\less2^{k |\alpha|}2^{\delta\,''(k-\ell)}\Phi_k^{n+2L+2\delta}(x-y).
\end{align*}
It is not crucial here that we took $\delta'<\delta''<\delta$, but this estimate will be used again later where our choice of $\delta'<\delta''$ will be important.  It follows that the kernel $T^*\psi_k^x(y)$ of $Q_kT$ satisfies
\begin{align*}
\left|\partial_y^\alpha T^*\psi_k^x(y)\right|&=2^{\ell|\alpha|}\left|\sum_{\ell>k}\<T((D^\alpha\psi)_{\ell}^{y}),\psi_k^x\>\right|\\
&\less2^{k|\alpha|}\sum_{\ell>k}2^{\delta\,''(k-\ell)}\Phi_k^{n+2L+2\delta}(x-y)\less2^{k|\alpha|}\Phi_k^{n+2L+2\delta}(x-y).
\end{align*}
This verifies that $T^*\psi_k^x(y)$ satisfies \eqref{sqker1} for $|x-y|\leq 2^{3-k}$.  We also verify the $\delta$-H\"older regularity estimate \eqref{sqker2} for $T^*\psi_k^x(y)$ with $\delta\,'$ in place of $\delta$:  let $\alpha\in\N_0^n$ with $|\alpha|=L$.  It trivially follows from the above estimate that
\begin{align*}
&\sum_{\ell\geq k:\;2^{-\ell}<|y-y'|}\left|\<D_1^\alpha T(\psi_{\ell}^{y}-\psi_{\ell}^{y'}),\psi_k^x\>\right|\\
&\hspace{1.3cm}\less\sum_{\ell\geq k:\;2^{-\ell}<|y-y'|}2^{\delta\,''(k-\ell)}(2^{\ell}|y-y'|)^{\delta\,'}2^{k|\alpha|}\(\Phi_k^{n+2L+2\delta}(x-y)+\Phi_k^{n+2L+2\delta}(x-y')\)\\
&\hspace{1.3cm}\less2^{k|\alpha|}\sum_{\ell\geq k\;2^{-\ell}<|y-y'|}2^{(\delta\,''-\delta\,')(k-\ell)}(2^k|y-y'|)^{\delta\,'}\(\Phi_k^{n+2L+2\delta}(x-y)+\Phi_k^{n+2L+2\delta}(x-y')\)\\
&\hspace{1.3cm}\less2^{k(|\alpha|+\delta')}|y-y'|^{\delta\,'}\(\Phi_k^{n+2L+2\delta}(x-y)+\Phi_k^{n+2L+2\delta}(x-y')\).
\end{align*}
On the other hand, for the situation where $|y-y'|\leq2^{-\ell}$, we consider
\begin{align*}
&\sum_{\ell\geq k:\;2^{-\ell}\geq|y-y'|}\left|\<D_1^\alpha T(\psi_{\ell}^{y}-\psi_{\ell}^{y'}),\psi_k^x\>\right|
\leq |A_{\ell,k}(x,y,y')|+|B_{\ell,k}(x,y,y')|,\\
\end{align*}
where
\begin{align*}
&A_{\ell,k}(x,y,y')=2^{\ell|\alpha|}\int_{|u-y|\leq2^{2-\ell}}T((D^\alpha\psi)_{\ell}^{y}-(D^\alpha\psi)_{\ell}^{y'})(u)\\
&\hspace{5cm}\times\(\psi_k^x(u)-\sum_{|\alpha|\leq L}\frac{D^\alpha\psi_k^x(y)}{\alpha!}(u-y)^\alpha\)du,\text{ and }\\
&B_{\ell,k}(x,y,y')=2^{\ell|\alpha|}\int_{|u-y|>2^{2-\ell}}T((D^\alpha\psi)_{\ell}^{y}-(D^\alpha\psi)_{\ell}^{y'})(u)\\
&\hspace{5cm}\times\(\psi_k^x(u)-\sum_{|\alpha|\leq L}\frac{D^\alpha\psi_k^x(y)}{\alpha!}(u-y)^\alpha\)du.
\end{align*}
The $A_{\ell,k}$ term is bounded as follows,
\begin{align*}
|A_{\ell,k}(x,y,y')|&\leq2^{\ell|\alpha|}||T((D^\alpha\psi)_{\ell}^{y}-(D^\alpha\psi)_{\ell}^{y'})\cdot\chi_{B(y,2^{1-\ell})}||_{L^1}\\
&\hspace{1.25cm}\times\left|\left|\(\psi_k^x(u)-\sum_{|\mu|\leq L}\frac{D^\mu\psi_k^x(y)}{\mu!}(u-y)^\mu\)\cdot\chi_{B(y,2^{1-\ell})}\right|\right|_{L^\infty}\\
&\hspace{0cm}\less2^{\ell |\alpha|}2^{-\ell n/2}||T((D^\alpha\psi)_{\ell}^{y}-(D^\alpha\psi)_{\ell}^{y'})||_{L^2}2^{(L+\delta)(k-\ell)}2^{kn}\\
&\hspace{0cm}\less2^{\ell|\alpha|}(2^\ell|y-y'|)^{\delta\,'}2^{(L+\delta)(k-\ell)}2^{kn}\\
&\hspace{0cm}\leq2^{k|\alpha|}2^{(\delta-\delta\,')(k-\ell)}(2^k|y-y'|)^{\delta\,'}\(\Phi_k^{n+2L+2\delta}(x-y)+\Phi_k^{n+2L+2\delta}(x-y')\).
\end{align*}
Recall the selection of $\delta\,''$ such that $0<\delta\,'<\delta\,''<\delta$.  The $B_{\ell,k}$ term is bounded using the kernel representation of $T$
\begin{align*}
|B_{\ell,k}(x,y,y')|&=2^{\ell|\alpha|}\left|\int_{|u-y|>2^{1-\ell}}\int_{\R^{n}}\(K(u,v)-\sum_{|\nu|\leq L}\frac{D_1^\nu K(u,y)}{\nu!}(v-y)^\nu\)\right.\\
&\hspace{1cm}\left.\times((D^\alpha\psi)_{\ell}^{y}(v)-(D^\alpha\psi)_\ell^{y'}(v))\(\psi_k^x(u)-\sum_{|\mu|\leq L}\frac{D^\mu\psi_k^x(y)}{\mu!}(u-y)^\mu\)du\,dv\right|\\
&\hspace{0cm}\less2^{\ell|\alpha|}\int_{|u-y|>2^{1-\ell}}\int_{\R^{n}}\frac{|v-y|^{L+\delta}}{|u-y|^{n+L+\delta}}\\
&\hspace{1cm}\times|(D^\alpha\psi)_{\ell}^{y}(v)-(D^\alpha\psi)_\ell^{y'}(v)|dv\,\,2^{kn}(2^k|u-y|)^{L+\delta''}du\\
&\hspace{0cm}\less2^{\ell|\alpha|}\sum_{m=1}^\infty\int_{2^{m-\ell}<|u-y|\leq2^{m+1-\ell}}\int_{\R^{n}}\frac{2^{-(L+\delta)\ell}}{2^{(n+L+\delta)(m-\ell)}}(2^{\ell}|y-y'|)^{\delta\,'}\\
&\hspace{1cm}\times\(\Phi_{\ell}^{n+1}(y-v)+\Phi_{\ell}^{n+1}(y'-v)\)dv\,2^{kn}(2^k|u-y|)^{L+\delta''}du\\
&\hspace{0cm}\less2^{\ell|\alpha|}\sum_{m=1}^\infty2^{n(m-\ell)}2^{-(L+\delta)\ell}2^{(n+L+\delta)(\ell-m)}(2^{\ell}|y-y'|)^{\delta\,'}2^{kn}2^{(L+\delta\,'')(k+m-\ell)}
\end{align*}
\begin{align*}
&\hspace{0cm}\less2^{k|\alpha|} 2^{(\ell-k)|\alpha|}2^{\delta'(\ell-k)}(2^{k}|y-y'|)^{\delta\,'}2^{kn}2^{(L+\delta\,'')(k-\ell)}\sum_{m=1}^\infty2^{(\delta\,''-\delta)m}\\
&\hspace{0cm}\less2^{k|\alpha|} (2^{k}|y-y'|)^{\delta\,'}2^{(\delta\,''-\delta\,')(k-\ell)}\(\Phi_k^{n+2L+2\delta}(x-y)+\Phi_k^{n+2L+2\delta}(x-y')\).
\end{align*}
It follows that
\begin{align*}
&\sum_{\ell=k}^\infty|A_{\ell,k}(x,y,y')|+|B_{\ell,k}(x,y,y')|\\
&\hspace{2cm}\less2^{k|\alpha|}(2^k|y-y'|)^{\delta'}\(\Phi_{k}^{n+2L+2\delta}(x-y)+\Phi_{k}^{n+2L+2\delta}(x-y')\)\sum_{\ell=k}^\infty2^{(\delta''-\delta')(k-\ell)}\\
&\hspace{2cm}\less2^{k|\alpha|}(2^k|y-y'|)^{\delta'}\(\Phi_{k}^{n+2L+2\delta}(x-y)+\Phi_{k}^{n+2L+2\delta}(x-y')\)
\end{align*}
We now check that $P_kT^*\psi_k^x(y)$, the second term from \eqref{furtherdecomp}, also satisfies the appropriate size and regularity estimates.  For all $\alpha\in\N_0^n$
\begin{align*}
|\partial_y^\alpha P_kT^*\psi_k^x(y)|&=2^{|\alpha|k}|\<T(D^\alpha\varphi)_k^y,\psi_k^x\>|\leq2^{|\alpha|k}||T||_{2,2}2^{kn}\less2^{|\alpha|k}\Phi_k^{n+2L+2\delta}(x-y).
\end{align*}
Here $||T||_{2,2}$ is the $L^2$ operator norm of $T$.  Therefore $T^*\psi_k^x(y)$ satisfies size and regularity properties \eqref{sqker1} and \eqref{sqker2} with $\delta'$ in place of $\delta$, and hence $\{Q_kT\}\in LPSO(n+2L+2\delta,L+\delta')$ for all $\delta'\in(0,\delta)$.  It is trivial now to note that for $\frac{n}{n+L+\delta}<p\leq1$, $T$ is bounded on $H^p$ if and only if $S_\Lambda$ is bounded from $H^p$ into $L^p$ since $||Tf||_{H^p}\approx ||S_\Lambda f||_{L^p}$ by the Littlewood-Paley-Stein characterization of $H^p$ in \cite{FS2}.
\end{proof}

\begin{lemma}\label{l:Txconditions}
Let $T\in CZO(M+\gamma)$ be bounded on $L^2$ and satisfy $T^*(x^\alpha)=0$ for all $|\alpha|\leq L=\lfloor M/2\rfloor$.  For $\psi\in\mathcal D_{M}$, define
\begin{align*}
d\mu_\psi(x,t)=\sum_{|\alpha|\leq L}\sum_{k\in\Z}|[[Q_kT]]_\alpha(x)|^2\delta_{t=2^{-k}}\,dx,
\end{align*}
where $Q_kf=\psi_k*f$ and $\psi_k(x)=2^{kn}\psi(2^kx)$.  If $[[T]]_\alpha\in BMO_{|\alpha|}$ for all $|\alpha|\leq L$, then $d\mu_\psi$ is a Carleson measure for any $\psi\in\mathcal D_{M+L}$.
\end{lemma}

\begin{proof}
Assume that $[[T]]_\alpha\in BMO_{|\alpha|}$ for all $|\alpha|\leq L$.  Let $\psi\in\mathcal D_{M+L}$, and it follows that $\{Q_kT\}\in LPSO(L,\delta')$ for all $\delta'<\delta$, where $Q_kf$ is defined as above and $L=\lfloor M/2\rfloor$ and $\delta=(M-2L+\gamma)/2$.  We also define $Q_k^\beta f=\psi_k^\beta*f$, where $\psi^\beta(x)=(-1)^{|\beta|}\psi(x)x^\beta$.  It follows that $\psi^\beta\in\mathcal D_{M+L-|\beta|}$.  Now let $\alpha\in\N_0^n$ such that $|\alpha|\leq L$.  Note that for $\beta\leq\alpha$, it follows that $\psi^\beta\in\mathcal D_M$, and hence $\{Q_k^\beta T\}\in LPSO(n+2L+2\delta,L+\delta')$ for all $0<\delta'<\delta$ as well.   Then it follows that
\begin{align*}
[[Q_kT]]_{\alpha}(x)&=2^{|\alpha|k}\int_{\R^{n}}T^*\psi_k^x(y)(x-y)^\alpha dy\\
&\hspace{0cm}=\lim_{R\rightarrow\infty}2^{|\alpha|k}\int_{\R^{2n}}\mathcal K(u,y)\psi_k^x(u)\eta_R(y)(x-y)^\alpha du\,dy\\
&\hspace{0cm}=\lim_{R\rightarrow\infty}\sum_{\beta\leq\alpha}c_{\alpha,\beta}2^{|\alpha|k}\int_{\R^{2n}}\mathcal K(u,y)\psi_k^x(u)(x-u)^\beta(u-y)^{\alpha-\beta} du\,dy\\
&\hspace{0cm}=\lim_{R\rightarrow\infty}\sum_{\beta\leq\alpha}c_{\alpha,\beta}2^{(|\alpha|-|\beta|)k}\int_{\R^{2n}}\mathcal K(u,y)(\psi_k^\beta)^x(u)\eta_R(y)(u-y)^{\alpha-\beta} du\,dy\\
&\hspace{0cm}=\sum_{\beta\leq\alpha}c_{\alpha,\beta}2^{(|\alpha|-|\beta|)k}\<[[T]]_{\alpha-\beta},(\psi_k^\beta)^x\>.
\end{align*}
Let $Q\subset\R^n$ be a cube with side length $\ell(Q)$.  It follows that
\begin{align*}
\sum_{2^{-k}\leq\ell(Q)}\int_Q|[[Q_kT]]_{\alpha}(x)|^2dx&\leq\sum_{2^{-k}\leq\ell(Q)}\int_Q\(\sum_{\beta\leq\alpha}c_{\alpha,\beta}2^{(|\alpha|-|\beta|)k}\left|\<[[T]]_{\alpha-\beta},(\psi_k^\beta)^x\>\right|\)^2dx\\
&\less\sum_{\beta\leq\alpha}\sum_{2^{-k}\leq\ell(Q)}\int_Q2^{2(|\alpha|-|\beta|)k}\left|\<[[T]]_{\alpha-\beta},(\psi_k^\beta)^x\>\right|^2dx\less|Q|.
\end{align*}
The last inequality holds since $[[T]]_{\alpha-\beta}\in BMO_{|\alpha|-|\beta|}$ and $\psi_k^\beta\in\mathcal D_M\subset\mathcal D_{|\alpha|-|\beta|}$ for all $\beta\leq\alpha$.
\end{proof}

Motivated by the proof of Lemma \ref{l:Txconditions}, we pause for a moment to introduce an alternative testing condition to $[[T]]_\alpha\in BMO_{|\alpha|}$ in Theorem \ref{t:CZbound}.  The following proposition introduces a perturbation of the definition of $[[T]]_\alpha$ with necessary and sufficient conditions for $[[T]]_\alpha\in BMO_{|\alpha|}$ for $|\alpha|\leq L$.

\begin{proposition}\label{p:equiv}
Let $T\in CZO(M+\gamma)$ with $T^*(y^\alpha)=0$ for $|\alpha|\leq L$.  Then $[[T]]_\alpha\in BMO_{|\alpha|}$ for all $|\alpha|\leq L$ if and only if
\begin{align*}
d\mu_\psi(x,t)=\sum_{|\alpha|\leq L}\sum_{k\in\Z}2^{2k|\alpha|}|\<TG_\alpha^x,\psi_k^x\>|^2\delta_{t=2^{-k}}\,dx
\end{align*}
is a Carleson measure for all $\psi\in\mathcal D_{M+L}$, where $Q_kf=\psi_k*f$ and $G_\alpha^x(u)=(u-x)^\alpha$.
\end{proposition}

The quantity $\<TG_\alpha^x,\psi\>$ is very closely related to $\<[[T]]_\alpha,\psi\>$.  One can obtain the distribution $TG_\alpha^x$ by replacing $(u-y)^\alpha$ with $(x-y)^\alpha$ in the definition of $[[T]]_\alpha$.  This gives an alternative testing condition for $[[T]]_\alpha\in BMO_{|\alpha|}$ that could be convenient in some situations.

\begin{proof}
Similar to the proof of Lemma \ref{l:Txconditions}, it follows that
\begin{align*}
2^{|\alpha|k}\<TG_\alpha^x,\psi_k^x\>&=\lim_{R\rightarrow\infty}2^{|\alpha|k}\int_{\R^{2n}}\mathcal K(u,y)\psi_k^x(u)\eta_R(y)(x-y)^\alpha du\,dy\\
&\hspace{0cm}=\sum_{\beta\leq\alpha}c_{\alpha,\beta}2^{(|\alpha|-|\beta|)k}\<[[T]]_{\alpha-\beta},(\psi_k^\beta)^x\>.
\end{align*}
Here $c_{\alpha,\beta}$ are binomial coefficients and are bounded uniformly for $|\alpha|,|\beta|\leq L$ depending on $L$.  Likewise we have that
\begin{align*}
2^{|\alpha|k}\<[[T]]_{\alpha},\psi_k^x\>
&\hspace{0cm}=\sum_{\beta\leq\alpha}c_{\alpha,\beta}2^{(|\alpha|-|\beta|)k}\<TG_{\alpha-\beta}^x,(\psi_k^\beta)^x\>.
\end{align*}
Lemma \ref{p:equiv} easily follows.
\end{proof}

Finally we prove Theorem \ref{t:reducedCZbound}.

\begin{proof}[Proof of Theorem \ref{t:reducedCZbound}]
By density, it is sufficient to prove the appropriate estimates for $f\in H^p\cap L^2$.  Let $\psi\in\mathcal D_{M+L}$ such that Calder\'on's reproducing formula \eqref{reproducingformula} holds for $Q_kf=\psi_k*f$, where $L=\lfloor M/2\rfloor$.  By Theorem \ref{t:CZtoLPS}, it follows that $\{\Lambda_k\}=\{Q_kT\}\in LPSO(n+2L+\delta,L+\delta')$ for all $0<\delta'<\delta=(M-2L+\gamma)/2$.  So fix a $\delta'\in(0,\delta)$ close enough to $\delta$ so that $\frac{n}{n+L+\delta}<\frac{n}{n+L+\delta'}<p$.  By Lemma \ref{l:Txconditions}, it follows that
\begin{align*}
d\mu(x,t)=\sum_{k\in\Z}\sum_{|\alpha|\leq L}|[[Q_kT]]_\alpha(x)|^2dx\,\delta_{t=2^{-k}}
\end{align*}
is a Carleson measure.  By Theorems \ref{t:sqbound} and \ref{t:CZtoLPS}, it also follows that $S_\Lambda$ can be extended to a bounded operator from $H^p$ into $L^p$, and hence $T$ can be extended to a bounded operator on $H^p$.
\end{proof}

\section{An Application to Bony Type Paraproducts}

In this section, we apply Theorem \ref{t:CZbound} to show that the Bony paraproduct operators from \cite{B} are bounded on $H^p$, which was stated in Theorem \ref{t:Bonyparaproduct}.  Let $\psi\in\mathcal D_{L+1}$ for some $L\geq0$ and $\varphi\in C_0^\infty$.  Define $Q_kf=\psi_k*f$ and $P_kf=\varphi_k*f$.  For $\beta\in BMO$, recall the definition of $\Pi_\beta$ in \eqref{paraproduct}
\begin{align*}
\Pi_\beta f(x)&=\sum_{j\in\Z}Q_j\(Q_j\beta\cdot P_jf\)(x).
\end{align*}
It follows that $\Pi_\beta\in CZO(M+\gamma)$ for all $M\geq0$ and $0<\gamma\leq1$.  We will focus on the properties $T^*(x^\alpha)=0$ and $[[T]]_\alpha\in BMO_{|\alpha|}$ for $|\alpha|\leq L$.  Once we prove these two things, we obtain Theorem \ref{t:Bonyparaproduct} by applying Theorem \ref{t:CZbound}.  We first give the definition of the Fourier transform that we will use and prove a lemma that will be used to prove the Hardy space bounds for $\Pi_\beta$.  For $f\in L^1(\R^n)$ and $\xi\in\R^n$, define
\begin{align*}
\widehat f(\xi)=\mathcal F[f](\xi)=\int_{\R^n}f(x)e^{ix\cdot\xi}dx.
\end{align*}

\begin{lemma}\label{l:conv}
Let $\psi\in\mathcal D_{M+1}$ for some integer $M$, and $-M\leq s\leq M$.  Define $V(x)$ and $V_k(x)$ by $\widehat V(\xi)=|\xi|^s\cdot\widehat\psi(\xi)$ and $V_k(x)=2^{kn}V(2^kx)$.  Also define
\begin{align*}
T_Vf(x)=\sum_{k\in\Z}V_k*f(x).
\end{align*}
Then $T_V$ is bounded on $H^1$ and on $BMO$.
\end{lemma}

\begin{proof}
We verify this lemma by showing that the convolution kernel of $T_V$ has uniformly bounded Fourier transform.  The kernel of $T_V$ is
\begin{align*}
K(x)=\sum_{k\in\Z}V_k(x).
\end{align*}
Then
\begin{align*}
|\widehat K(\xi)|\leq\sum_{k\in\Z}|\widehat V(2^{-k}\xi)|&=\sum_{k\in\Z}(2^{-k}|\xi|)^s|\widehat \psi(2^{-k}\xi)|^2\\
&\less\sum_{k\in\Z}(2^{-k}|\xi|)^s\min(2^{-k}|\xi|,2^k|\xi|^{-1})^{M+1}\\
&\less\sum_{k\in\Z}\min(2^{-k}|\xi|,2^k|\xi|^{-1})\less1.
\end{align*}
Note that since $\psi\in\mathcal D_{M+1}$, it follows that $|\widehat\psi(\xi)|\leq\min(|\xi|,|\xi|^{-1})^{M+1}$.  It follows that $T_V$ is bounded on $H^1$ and on $BMO$; see \cite{FS2}.
\end{proof}

\begin{proof}[Proof of Theorem \ref{t:Bonyparaproduct}]
As remarked above, it is clear that $\Pi_\beta\in CZO(M+\gamma)$ for all $M\geq0$ and $0<\gamma\leq1$.  So it is enough to show that $T^*(x^\alpha)=0$ and $[[T]]_\alpha\in BMO_{|\alpha|}$ for $|\alpha|\leq L$.  For $f\in\mathcal D_L$, we check the first condition.
\begin{align*}
\<\Pi_\beta^*(x^\alpha),f\>&=\lim_{R\rightarrow\infty}\sum_{j\in\Z}\<Q_j\(Q_j\beta\cdot P_jf\),\eta_R\cdot x^\alpha\>\\
&=\lim_{R\rightarrow\infty}\sum_{j\in\Z}\int_{\R^n}Q_j\beta(u) P_jf(u)Q_j(\eta_R\cdot x^\alpha)(u)du\\
&=\sum_{j\in\Z}\int_{\R^n}Q_j\beta(u) P_jf(u)Q_j(x^\alpha)(u)du=0
\end{align*}
since $Q_j(x^\alpha)=0$ for $|\alpha|\leq L$.  We also verify the $BMO_{|\alpha|}$ conditions.  Let $|\alpha|\leq L$, and compute
\begin{align*}
\<[[\Pi_\beta]]_{\alpha},\psi_k^x\>&=\lim_{R\rightarrow\infty}\sum_{j\in\Z}\int_{\R^{2n}}\psi_j(u-v)Q_j\beta(v)\(\int_{\R^n} \varphi_j(v-y)(u-y)^\alpha \eta_R(y)dy\)\psi_k^x(u)dv\,du\\
&=\sum_{\mu\leq\alpha}c_{\alpha,\mu}\lim_{R\rightarrow\infty}\sum_{j\in\Z}\int_{\R^{2n}}\psi_j(u-v)Q_j\beta(v)\\
&\hspace{3.3cm}\times\(\int_{\R^n} \varphi_j(v-y)(u-v)^{(\mu)}(v-y)^{\alpha-\mu} \eta_R(y)dy\)\psi_k^x(u)dv\,du\\
&=\sum_{\mu\leq\alpha}c_{\alpha,\mu}C_{\alpha-\mu}\sum_{j\in\Z}2^{-|\alpha|j}\int_{\R^{2n}}\psi_j^{(\mu)}(u-v)Q_j\beta(v)\psi_k^x(u)dv\,du\\
&=\sum_{\mu\leq\alpha}c_{\alpha,\mu}C_{\alpha-\mu}\sum_{j\in\Z}2^{-|\alpha|j}Q_kQ_j^{(\mu)} Q_j\beta(x),
\end{align*}
where $\psi^{(\mu)}(x)=x^\mu\,\psi(x)$, $\psi_j^{(\mu)}(x)=2^{jn}\psi^{(\mu)}(2^jx)$, and $Q_j^{(\mu)}f(x)=\psi_j^{(\mu)}*f(x)$.  Now we consider
\begin{align*}
2^{|\alpha|(k-j)}\mathcal F\[Q_kQ_j^{(\mu)} Q_jf\](\xi)&=2^{|\alpha|(k-j)}\widehat\psi(2^{-k}\xi)\widehat{\psi^{(\mu)}}(2^{-j}\xi)\widehat{\psi}(2^{-j}\xi) \widehat f(\xi)\\
&=\((2^{-k}|\xi|)^{-|\alpha|}\widehat\psi(2^{-k}\xi)\)\((2^{-j}|\xi|)^{|\alpha|}\widehat{\psi^{(\mu)}}(2^{-j}\xi)\widehat{\psi}(2^{-j}\xi)\) \widehat f(\xi)\\
&=\mathcal F\[W_k*V_j*f\](\xi),
\end{align*}
where $W$ and $V$ are defined by $\widehat W(\xi)=|\xi|^{-|\alpha|}\widehat\psi(\xi)$, $\widehat {V^{(\mu)}}(\xi)=|\xi|^{|\alpha|} \widehat{\psi^{(\mu)}}(\xi)\widehat\psi(\xi)$, $W_k(x)=2^{kn}W(2^kx)$, and $V_j^{(\mu)}(x)=2^{jn}V^{(\mu)}(2^jx)$.  Here $c_{\alpha,\mu}$ are binomial coefficients, and $C_\mu=\int_{\R^n}\varphi(x)x^\mu dx$.  By Lemma \ref{l:conv}, it follows that
\begin{align*}
T_{V^{(\mu)}}f(x)=\sum_{j\in\Z}V_j^{(\mu)}*f(x)
\end{align*}
defines an operator that is bounded on $BMO$.  Then
\begin{align*}
\sum_{j\in\Z}2^{|\alpha|(k-j)}Q_kQ_j^{(\mu)} Q_j\beta(x)&=\sum_{j\in\Z}W_k*V_j^{(\mu)}*\beta(x)=W_k*(T_{V^{(\mu)}}\beta)(x),
\end{align*}
and we have the following
\begin{align*}
\int_Q\sum_{2^{-k}\leq\ell(Q)}2^{2|\alpha|k}|\<[[\Pi_\beta]]_{\alpha},\psi_k^x\>|^2&=\int_Q\sum_{2^{-k}\leq\ell(Q)}\left|\sum_{\mu\leq\alpha}c_{\alpha,\mu}C_{\alpha-\mu}\sum_{j\in\Z}2^{|\alpha|(k-j)}Q_kQ_j^{(\mu)} Q_j\beta(x)\right|^2\\
&\less\sum_{\mu\leq\alpha}|c_{\alpha,\mu}C_{\alpha-\mu}|^2\int_Q\sum_{2^{-k}\leq\ell(Q)}\left|W_k*(T_{V^{(\mu)}}\beta)(x)\right|^2.
\end{align*}
Note that $|\widehat W(\xi)|\less\min(|\xi|,|\xi|^{-1})$ as well, and since $T_{V^{(\mu)}}\beta\in BMO$ with $||T_{V^{(\mu)}}\beta||\less||\beta||_{BMO}$, it also follows that
\begin{align*}
\frac{1}{|Q|}\int_Q\sum_{2^{-k}\leq\ell(Q)}2^{2|\alpha|k}|\<[[\Pi_\beta]]_{\alpha},\psi_k^x\>|^2&\less\sum_{\mu\leq\alpha}|c_{\alpha,\mu}C_{\alpha-\mu}|^2\int_Q\sum_{2^{-k}\leq\ell(Q)}\left|W_k*(T_{V^{(\mu)}}\beta)(x)\right|^2\\
&\less||T_{V^{(\alpha)}}\beta||_{BMO}^2\less||\beta||_{BMO}^2.
\end{align*}
Therefore $[[\Pi_\beta]]_\alpha\in BMO_{|\alpha|}$ for $|\alpha|\leq L$, and by Theorem \ref{t:CZbound} it follows that $\Pi_\beta$ is bounded on $H^p$ for all $\frac{n}{n+L+\delta}<p\leq1$, where $L=\lfloor M/2\rfloor$ and $\delta=(M-2L+1)/2$.
\end{proof}

\section{Proof of Theorem \ref{t:CZbound}}

Finally, we return to the proof of Theorem \ref{t:CZbound}.  We have waited to this point to do so since we will need both Theorem \ref{t:reducedCZbound} and the Bony paraproduct construction in Theorem \ref{t:Bonyparaproduct}.

We need one other result from \cite{T,FTW,FHJW}; we state Theorem 3.13 from \cite{FTW} adapted to our notation and restricted to the Hardy space setting.

\begin{theorem}[\cite{FTW}]
Let $T\in CZO(M+\gamma)$ be bounded on $L^2$ and define $L=\lfloor M/2\rfloor$ and $\delta=(M-2L+\gamma)/2$.  If $T^*(x^\alpha)=0$ in $\mathcal D_M'$ for all $|\alpha|\leq L$ and $T1=0$ in $\mathcal D_0$, then $T$ is bounded on $H^p$ for all $\frac{n}{n+L+\delta}<p\leq1$.
\end{theorem}

In the notation of \cite{FTW}, this theorem is stated with $q=2$, $0<p\leq1$, $J=n/p$, $L=\lfloor J-n\rfloor=\lfloor n/p-n\rfloor$, $\alpha=0$, and $H^p=\dot F_p^{0,2}$.

\begin{proof}[Proof of Theorem \ref{t:CZbound}]\label{t:FTW}
Let $T\in CZO(M+\gamma)$ be bounded on $L^2$ and define $L=\lfloor M/2\rfloor$ and $\delta=(M-2L+\gamma)/2$.  Assume that $T^*(x^\alpha)=0$ in $\mathcal D_M'$ for all $|\alpha|\leq L$.  Then $T1\in BMO$, and by Theorem \ref{t:Bonyparaproduct} there exists $\Pi\in CZO(M+1)$ such that $\Pi(1)=T(1)$, $\Pi^*(y^\alpha)=0$ for $|\alpha|\leq M$, and $\Pi$ is bounded on $H^p$ for all $\frac{n}{n+L+1}<p\leq1$.  Then $T=S+\Pi$, where $S=T-\Pi$.  Noting that $S^*(y^\alpha)=0$ for all $|\alpha|\leq L$ and $S1=0$, by Theorem \ref{t:FTW} it follows that $S$ is bounded on $H^p$ for all $\frac{n}{n+L+\delta}$.  Therefore $T$ is bounded on $H^p$ for all $\frac{n}{n+L+\delta}<p\leq1$.

Now assume that $T$ is bounded on $H^p$ for all $\frac{n}{n+L+\delta}<p\leq1$.  For $\psi\in\mathcal D_L$, it follows that $T\psi\in H^p\cap L^2$ for all $\frac{n}{n+L+\delta}<p\leq1$.  It is not hard to show that
\begin{align*}
\int_{\R^n}T\psi(x)x^\alpha dx
\end{align*}
is an absolutely convergent integral for any $|\alpha|<\sup\{n/p-n:\frac{n}{n+L+\delta}<p\leq1\}=L+\delta$.  By Theorem 7 in \cite{GH}, it follows that
\begin{align*}
\int_{\R^n}T\psi(x)x^\alpha dx=0
\end{align*}
for all $\alpha\in\N_0^n$ with $|\alpha|<L+\delta$.  Since $\delta>0$, this verifies that $T^*(y^\alpha)=0$ for all $|\alpha|\leq L$.
\end{proof}

\bibliographystyle{amsplain}

\end{document}